\documentstyle[12pt,figure]{article}
\begin{document}

%blackboard bold for 10 pt
%\font\bbbld=msym10
%blackboard bold for 11 pt
%\font\bbbld=msym10 scaled\magstephalf
%blackboard bold for 12 pt
\font\bbbld=msbm10 scaled\magstep1
\newcommand{\bfR}{\hbox{\bbbld R}}
\newcommand{\bfC}{\hbox{\bbbld C}}
\newcommand{\bfZ}{\hbox{\bbbld Z}}
\newcommand{\bfH}{\hbox{\bbbld H}}
\newcommand{\bfQ}{\hbox{\bbbld Q}}
\newcommand{\bfN}{\hbox{\bbbld N}}
\newcommand{\bfP}{\hbox{\bbbld P}}
\newcommand{\bfT}{\hbox{\bbbld T}}
\def\Sym{\mathop{\rm Sym}}
\newcommand{\halo}[1]{\Int(#1)}
\def\Int{\mathop{\rm Int}}
\def\Re{\mathop{\rm Re}}
\def\Im{\mathop{\rm Im}}
\newcommand{\union}{\cup}
\newcommand{\goesto}{\rightarrow}
\newcommand{\bdy}{\partial}
\newcommand{\n}{\noindent}
\newcommand{\p}{\hspace*{\parindent}}

\newtheorem{theorem}{Theorem}[section]
\newtheorem{assertion}{Assertion}[section]
\newtheorem{proposition}{Proposition}[section]
\newtheorem{lemma}{Lemma}[section]
\newtheorem{definition}{Definition}[section]
\newtheorem{claim}{Claim}[section]
\newtheorem{corollary}{Corollary}[section]
\newtheorem{observation}{Observation}[section]
\newtheorem{conjecture}{Conjecture}[section]
\newtheorem{question}{Question}[section]
\newtheorem{example}{Example}[section]

\newbox\qedbox
\setbox\qedbox=\hbox{$\Box$}
\newenvironment{proof}{\smallskip\noindent{\bf Proof.}\hskip \labelsep}%
                        {\hfill\penalty10000\copy\qedbox\par\medskip}
\newenvironment{remark}{\smallskip\noindent{\bf Remark.}\hskip \labelsep}%
                        {\hfill\penalty10000\copy\qedbox\par\medskip}
\newenvironment{remark1}{\smallskip\noindent{\bf Remark 1.}\hskip \labelsep}%
                        {\hfill\penalty10000\copy\qedbox\par\medskip}
\newenvironment{remark2}{\smallskip\noindent{\bf Remark 2.}\hskip \labelsep}%
                        {\hfill\penalty10000\copy\qedbox\par\medskip}
\newenvironment{proofspec}[1]%
                      {\smallskip\noindent{\bf Proof of Theorem 1.1.}
                        \hskip \labelsep}%
                        {\nobreak\hfill\hfill\nobreak\copy\qedbox\par\medskip}
\newenvironment{proofspec2}[1]%
                      {\smallskip\noindent{\bf Proof of Theorem 1.2.}
                        \hskip \labelsep}%
                        {\nobreak\hfill\hfill\nobreak\copy\qedbox\par\medskip}
\newenvironment{acknowledgements}{\smallskip\noindent{\bf Acknowledgements.}%
        \hskip\labelsep}{}

\setlength{\baselineskip}{1.0\baselineskip}

\title{Constant Mean Curvature Surfaces with Two Ends in Hyperbolic Space}
\author{Wayne Rossman, Katsunori Sato}

\maketitle

\begin{abstract}
We investigate the close relationship between 
minimal surfaces in Euclidean 3-space and constant mean 
curvature 1 surfaces in hyperbolic 3-space.  Just as in 
the case of minimal surfaces in Euclidean 3-space, the only complete 
connected embedded constant mean curvature 1 surfaces with two ends in 
hyperbolic space are well-understood surfaces of revolution -- the catenoid 
cousins.  

In contrast to this, we 
show that, unlike the case of minimal surfaces in Euclidean 
3-space, there do exist complete 
connected immersed constant mean curvature 1 surfaces with two ends in 
hyperbolic space that are not surfaces of revolution -- the genus 1 catenoid 
cousins.  The genus 1 catenoid cousins are of interest because 
they show that, although minimal surfaces in Euclidean 
3-space and constant mean curvature 1 surfaces in hyperbolic 3-space are 
intimately related, there are 
essential differences between these two sets of surfaces.  
The proof we give of existence of the 
genus 1 catenoid cousins is a mathematically rigorous 
verification that the results of a computer experiment are sufficiently 
accurate to imply existence.
\end{abstract}

\section{Introduction}

The main result presented in this paper is motivated primarily by 
a result of Schoen \cite{S}, that the only complete connected 
finite-total-curvature minimal immersions in
$\bfR^3$ with two embedded ends are 
catenoids.  In this paper we investigate the closely related 
case of constant mean curvature (CMC) 1 surfaces with two ends in
hyperbolic space $\bfH^3$.  Other motivations 
are the results of Kapouleas, Korevaar, Kusner, Meeks, and Solomon.
In \cite{KKS} it was shown that any complete
properly embedded nonminimal CMC surface with 
two ends in $\bfR^3$ is a periodic surface of revolution (Delaunay 
surface).
In \cite{K} it was shown that there exist immersed complete 
nonminimal CMC surfaces with two ends in $\bfR^3$ with genus $g \geq 2$.
And in \cite{KKMS} it was shown that any complete
properly embedded CMC $c$ ($c > 1$) surface with
two ends in $\bfH^3$ is a periodic surface of revolution 
(hyperbolic Delaunay surface).

CMC 1 surfaces in $\bfH^3$ are closely 
related to minimal
surfaces in $\bfR^3$.  There is a natural correspondence between
them, known as Lawson's correspondence.  
Let $\cal U$ be a simply-connected region of $\bfC$.
If $x : {\cal U} \rightarrow
\bfR^3$ is a local representation for a minimal surface in $\bfR^3$
with first and second fundamental forms $I$ and $II$, then 
we see (via the Gauss and Codazzi equations and 
the fundamental theorem for surfaces) that there is a well-defined
CMC 1 surface 
$\tilde{x} : {\cal U} \subset \bfC \rightarrow
\bfH^3$ with first and second fundamental forms $I$ and $II + I$.
In addition, CMC 1 surfaces in $\bfH^3$ have a
Weierstrass representation based on a pair of holomorphic functions 
\cite{B},
similar to the Weierstrass representation for minimal surfaces in
$\bfR^3$ (see section 2). 

The following theorem was first proven by Levitt and Rosenberg, 
and holds for any CMC $c$ surface, not just for surfaces 
satisfying $c=1$.  We include a proof here (section 3) for the sake of 
completeness.  

\begin{theorem}
 \cite{LR} Any complete properly embedded 
CMC $c$ surface in $\bfH^3$ with asymptotic boundary 
consisting of at most two points is a surface of revolution.  In
particular, it is homeomorphic to a punctured sphere.
\end{theorem}

In the case that $c=1$, this theorem implies that the surface must be
a genus 0 catenoid cousin.  (This was shown in \cite{UY1}.  The genus 0 catenoid 
cousins were originally described in \cite{B}.)  The condition that the surface 
has asymptotic boundary at most two points implies that $c \geq 1$, 
as shown by do Carmo, Gomes, and Thorbergsson \cite{CGT}.  

We will show that the condition
``embedded'' is critical to the above theorem, by giving an immersed 
counterexample, which we call the genus 1 catenoid cousin. 

\begin{theorem} 
There exists a one-parameter family of CMC 1 
genus 1 complete properly immersed surfaces in $\bfH^3$ with asymptotic 
boundary consisting of two points.
\end{theorem}

The genus 1 catenoid cousin 
demonstrates a clear difference between CMC 1 surfaces in $\bfH^3$ 
and minimal surfaces in $\bfR^3$, since Schoen's result on minimal surfaces 
in $\bfR^3$ holds even for immersions.  

The genus 1 catenoid cousin further shows that the set of CMC 1 surfaces 
in $\bfH^3$ with embedded ends is in some sense larger than the set
of minimal surfaces in $\bfR^3$ with embedded ends.  Loosely speaking, 
the set of complete minimal surfaces with embedded ends in $\bfR^3$ can 
be mapped injectively to a set of (one-parameter families of) 
corresponding complete CMC 1 surfaces with embedded ends 
in $\bfH^3$ (\cite{RUY}).  The second theorem above shows that
we cannot map the set of (one-parameter families of) 
complete CMC 1 surfaces with embedded 
ends in $\bfH^3$ injectively to a set of 
corresponding complete minimal surfaces with embedded ends 
in $\bfR^3$, since there does not exist a minimal surface in 
$\bfR^3$ which corresponds to the genus 1 catenoid cousin in $\bfH^{3}$.

In section 4 we give a nonrigorous explanation for why one should 
expect the genus 1 catenoid cousins to exist.  
The remainder of the paper is then devoted to proving Theorem 1.2 
rigorously.  We 
believe that the proof has two interesting characteristics:  

\begin{itemize}
\item One is 
that the period problems that must be solved can be reduced to a 
single period problem, using symmetry properties of the surface, by 
a fairly direct argument.  
This kind of dimension reduction of the period problem can usually be 
done in a geometric and uncomplicated way for minimal surfaces is 
$\bfR^3$, but for CMC 1 surfaces in $\bfH^3$ this kind of dimension 
reduction seems to be inherently more algebraic and less 
geometrically transparent \cite{RUY}.  
\item The other is that we then solve the single remaining period 
numerically, and then we use a mathematically rigorous analysis of 
the numerical method to conclude that the numerical 
results are correct.  Certainly, this kind of 
``numerical error analysis'' 
has been used before, for example in \cite{HHS} and \cite{KPS}, and 
it is likely to be used frequently in the future, as it is well 
suited for solving period problems on surfaces for which no other 
method of solution can be found.  It is easy to imagine how 
this method could be useful in a very wide variety of situations.  
\end{itemize}

We solve the single period problem by applying the intermediate value 
theorem.  The idea is similar to the way the intermediate value 
theorem is used in the conjugate Plateau construction to solve period 
problems for minimal surfaces in $\bfR^3$ \cite{Kar}, \cite{BR}.  
However, the conjugate Plateau construction fails to help 
us in the study of CMC 1 surfaces in $\bfH^3$, hence we have used 
numerical analysis instead.  (The conjugate Plateau construction 
{\em is} of use in studying {\em minimal} surfaces in $\bfH^3$ 
\cite{P}, but does not appear to be useful for studying CMC 1 
surfaces in $\bfH^3$.)  

The same methods we use here could likely also be applied to produce 
similar examples with two ends and genus greater than one, without any 
conceptual additions.  However, with genus greater than 1, 
after reducing the period problems to a minimal set, we would still 
have at least a 2-dimensional problem, and thus the computational 
aspects would become much more involved.  As the genus 1 
example fulfills our goal of finding a counterexample to Schoen's 
result in the hyperbolic case (and is computationally more easily 
understandable), we felt it was appropriate to restrict ourselves to 
genus 1.  

Although this paper is written from a mathematical viewpoint, the 
arguments used here became apparent to the authors only 
by means of a numerical experiment.  Hence, from the authors' point of 
view, experimental results were essential in obtaining 
the above result.  

\begin{figure}
        \hspace{1.51in}
        \epsfxsize=3.2in
        \epsffile{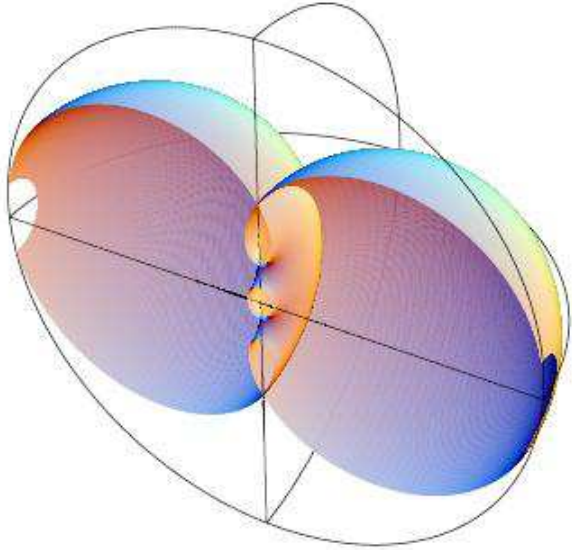}
        \hfill
\caption{A genus 1 catenoid cousin in 
the Poincare model for $\bfH^3$.  (Half of the surface has been cut away.)}
%On the left side is a compact central portion of the surface.  
%On the right side is the entire surface, and a small schematic of the central 
%cross-section.  }
\end{figure}

\section{The Weierstrass representation}

Both minimal surfaces in $\bfR^3$ and CMC 
1 surfaces in 
$\bfH^3$ can
be described parametrically by a pair of meromorphic functions on a 
Riemann surface, via a Weierstrass representation.  First we describe 
the well-known Weierstrass representation for minimal surfaces in $\bfR^3$.
We will incorporate into this representation 
the fact that any complete minimal surface of 
finite total curvature is 
conformally equivalent to a Riemann surface 
$\Sigma$ with a finite number of 
points $\{p_j\}_{j=1}^k \subset 
\Sigma$ removed (\cite{O}):  

\begin{lemma}
Let $\Sigma$ be a Riemann surface.  Let $\{p_j\}_{j=1}^k \subset 
\Sigma$ be a finite number of points, which will represent the ends 
of the minimal surface defined in this lemma.  
Let $z_0$ be a fixed point in $\Sigma \setminus \{p_j\}$.
Let $g$ be a meromorphic function from $\Sigma \setminus \{p_j\}$ to 
the complex plane $\bfC$.  
Let $f$ be a holomorphic function from $\Sigma \setminus \{p_j\}$ 
to $\bfC$.  Assume 
that, for any point in $\Sigma \setminus \{p_j\}$, $f$ has a zero of 
order $2k$ at some point if and only if $g$ has 
a pole of order $k$ at that point, and assume that $f$ has no other 
zeroes on $\Sigma \setminus \{p_j\}$.  Then 
\[ \Phi(z) = \mbox{Re} \int_{z_0}^{z}
\; \left( \begin{array}{c}
        (1-g^2)f d\zeta \\
        i(1+g^2)f d\zeta \\
        2gf d\zeta
        \end{array}
\right) \] is a conformal minimal immersion of the universal cover 
$\widetilde{\Sigma \setminus \{p_j\}}$ of 
$\Sigma \setminus \{p_j\}$ into $\bfR^3$. 
Furthermore, any complete 
minimal surface in $\bfR^3$ can be represented 
in this way.
\end{lemma}

The map $g$ can be geometrically interpreted as the stereographic 
projection of the Gauss map.  
The first and second fundamental forms and the intrinsic Gaussian 
curvature for the surface $\Phi(z)$ are 
    \[ ds^2 = (1+g\bar g)^2\,f \bar{f} dz \overline{dz} \; , \; \; 
        II = -2\mbox{Re}(Q) \; , \; \; 
   K = -4 \left( \frac{|g^\prime|}{|f| (1+|g|^2)^2} \right)^2  \; , \]
where $Q = f g^\prime dz^2$ is the Hopf differential.  

To make a surface of finite total curvature $\int_\Sigma -K dA 
< + \infty$, we must choose $f$ and
$g$ so that $\Phi$ is well defined on $\Sigma 
\setminus \{p_j\}$ itself.  Usually this 
involves ajusting some real parameters in the descriptions of $f$ 
and $g$ and $\Sigma \setminus \{p_j\}$ so that the 
real part of the above integral about any
nontrivial loop in $\Sigma \setminus \{p_j\}$ is zero.

We now describe a Weierstrass type representation for 
CMC $c$ surfaces in $\bfH^3(-c^2)$.  ($\bfH^3(-c^2)$ is a 
simply-connected complete 3-dimensional space with constant sectional 
curvature $-c^2$.  $\bfH^3 := \bfH^3(-1)$.)  
This lemma is a composition of results found in \cite{B}, 
\cite{UY3}, and \cite{UY4}.  

\begin{lemma}
Let $\Sigma$, $\Sigma \setminus \{p_j\}$, 
$z_0$, $f$, and $g$ be the same as in the previous 
lemma.  Choose a null holomorphic immersion $F:
\widetilde{\Sigma \setminus \{p_j\}} \to SL(2,\bfC)$ so 
that $F(z_0)$ is the identity
matrix and so that $F$ satisfies 
\begin{equation}
F^{-1}dF = c \left( \begin{array}{cc} 
g & -g^2 \\ 
1 & -g 
\end{array} \right) f dz \; , 
    \label{eq:wode}
\end{equation}
then 
$\Phi:\widetilde{\Sigma \setminus \{p_j\}} \to H^3(-c^2)$ defined by 
\begin{equation}
        \Phi = \frac{1}{c} F^{-1} \overline{F^{-1}}^t
    \label{eq:imm}
\end{equation}
is a conformal CMC $c$ 
immersion into $\bfH^3(-c^2)$ with the Hermitean model.  
Furthermore, any CMC $c$ surface 
in $\bfH^3(-c^2)$ can be represented in this way.
\end{lemma}

A description of the Hermitean model can be found in
any of \cite{B}, \cite{UY1}, and \cite{UY2}, but we also briefly 
describe it here.  If ${\cal L}^4$ denotes the standard 
Lorentzian 4-space of signature $(-+++)$, then the Minkowsky model for 
$\bfH^3(-c^2)$ is $\bfH^3(-c^2) = 
\{ (t,x_1,x_2,x_3) \in {\cal L}^4 \; : \; \sum_{j=1}^3 x_j^2 - t^2 = 
\frac{-1}{c^2} \}$.  We can identify each point $(t,x_1,x_2,x_3)$ in the 
Minkowsky model with a point 
\[  \left(\begin{array}{cc}
                            t+x_3 & x_1+ix_2 \\
                    x_1-ix_2 & t-x_3
                \end{array}\right) \]
in the space of $2 \times 2$ Hermitean matrices.  Thus the Hermitean model for 
$\bfH^3(-c^2)$ is \[ \bfH^3(-c^2) = 
\{ \pm \frac{1}{c} a \cdot \bar{a}^t \; : \; a \in SL(2,\bfC) \} . \]  

We call $g$ the {\it hyperbolic Gauss map} of $\Phi$.   
As its name suggests, 
the map $g(z)$ has a geometric interpretation for this case as well.  
It is the image
of the composition of two maps.  The first map is from each point $z$ on 
the surface to the point at the sphere at infinity in the Poincare 
model
which is at the opposite 
end of the oriented perpendicular geodesic ray starting at 
$z$ on the surface.  The second map is stereographic projection of the
sphere at infinity to the complex plane $\bfC$ \cite{B}.
The first and second fundamental 
forms and the intrinsic Gaussian curvature 
of the surface are 
    \[
        ds^2 = (1+G\bar G)^2 \, \frac{f g^\prime}{G^\prime} 
         \overline{\left(\frac{f g^\prime}{G^\prime}\right)} dz 
         \overline{dz} \; , \; \; 
        II = -2 \mbox{Re}(Q) + c\,ds^2 \; , \; \; 
   K = -4 \left( \frac{|G^\prime|^2}{|g^\prime| 
|f| (1+|G|^2)^2} \right)^2  \; ,
    \label{eq:second}
    \] where in this case the Hopf differential is $Q = - f g^\prime 
dz^2$ (the sign change in $Q$ is due to the fact that we are 
considering the ``dual'' surface; see \cite{UY4} for an 
explanation of this), and 
where $G$ is defined as the multi-valued meromorphic function 
\[
G=\frac {dF_{11}}{dF_{21}}=\frac {dF_{12}}{dF_{22}} 
\]
on $\Sigma \setminus \{p_j\}$, 
with $F=(F_{ij})_{i,j=1,2}$.  The reason that $G$ is multi-valued is 
that $F$ itself can be multi-valued on $\Sigma \setminus \{p_j\}$ 
(even if $\Phi$ is well defined on $\Sigma \setminus \{p_j\}$ itself).
The function $G$ is called the {\it secondary Gauss map}
of $\Phi$ (\cite{B}).  

In the above lemma, we have 
changed the notation slightly from the notation used in $\cite{B}$, 
because we wish to use the same symbol ``$g$'' both for the map $g$ 
for minimal surfaces in 
$\bfR^3$ and for the hyperbolic Gauss map for CMC $c$ surfaces in 
$\bfH^3(-c^2)$.  And we further wish to give a 
separate notation ''$G$'' for 
the secondary Gauss map used in the hyperbolic case.  We do this to 
emphasize that the ``$g$'' in the Euclidean case 
is more closely related to the hyperbolic Gauss map 
``$g$'' in the $\bfH^3$ case than 
to the geometric Gauss map ``$G$''.  

\begin{figure}
        \hspace{1.51in}
        \epsfxsize=1.8in
        \epsffile{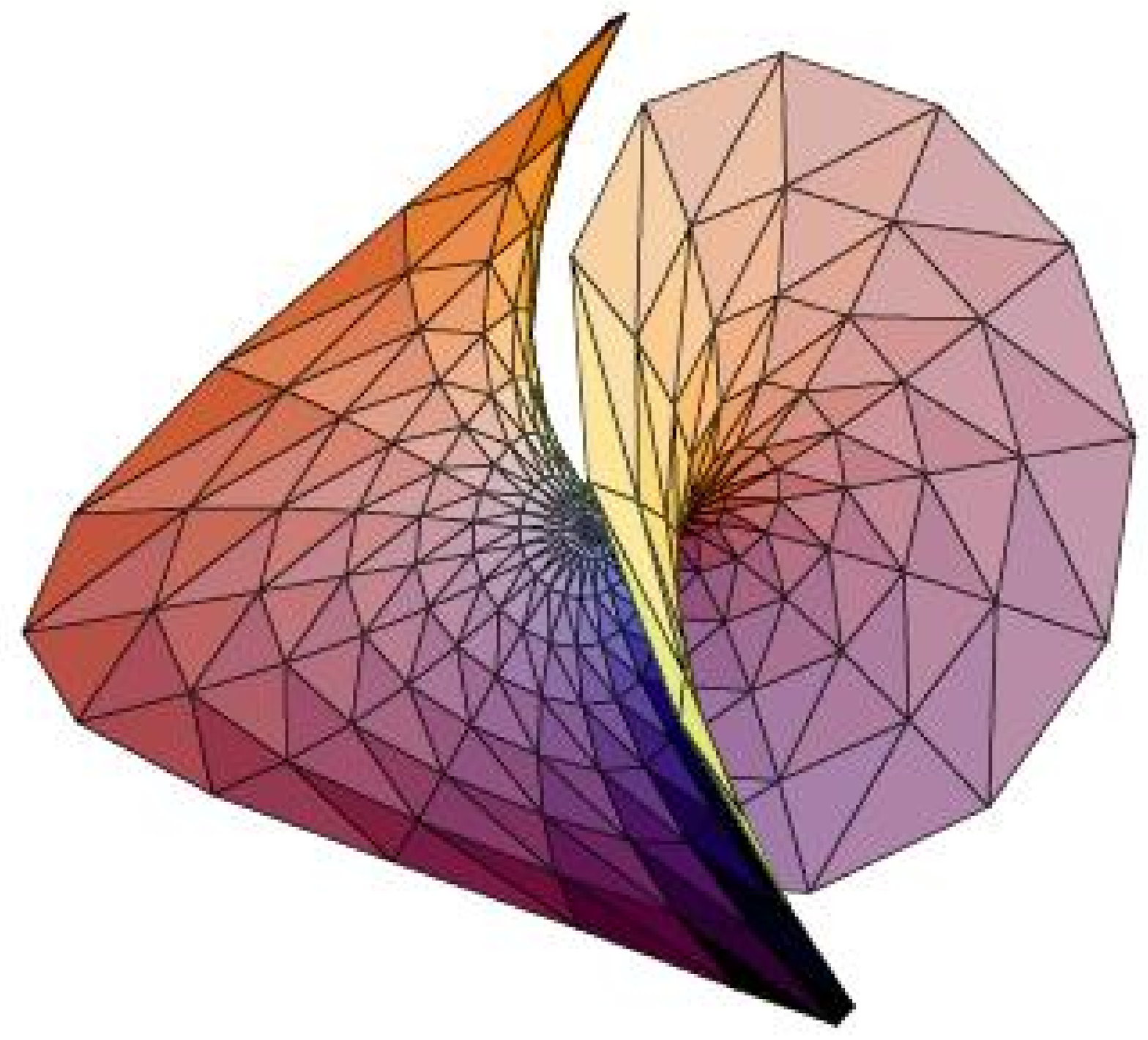}
        \hspace{0.01in}
        \epsfxsize=1.8in
        \epsffile{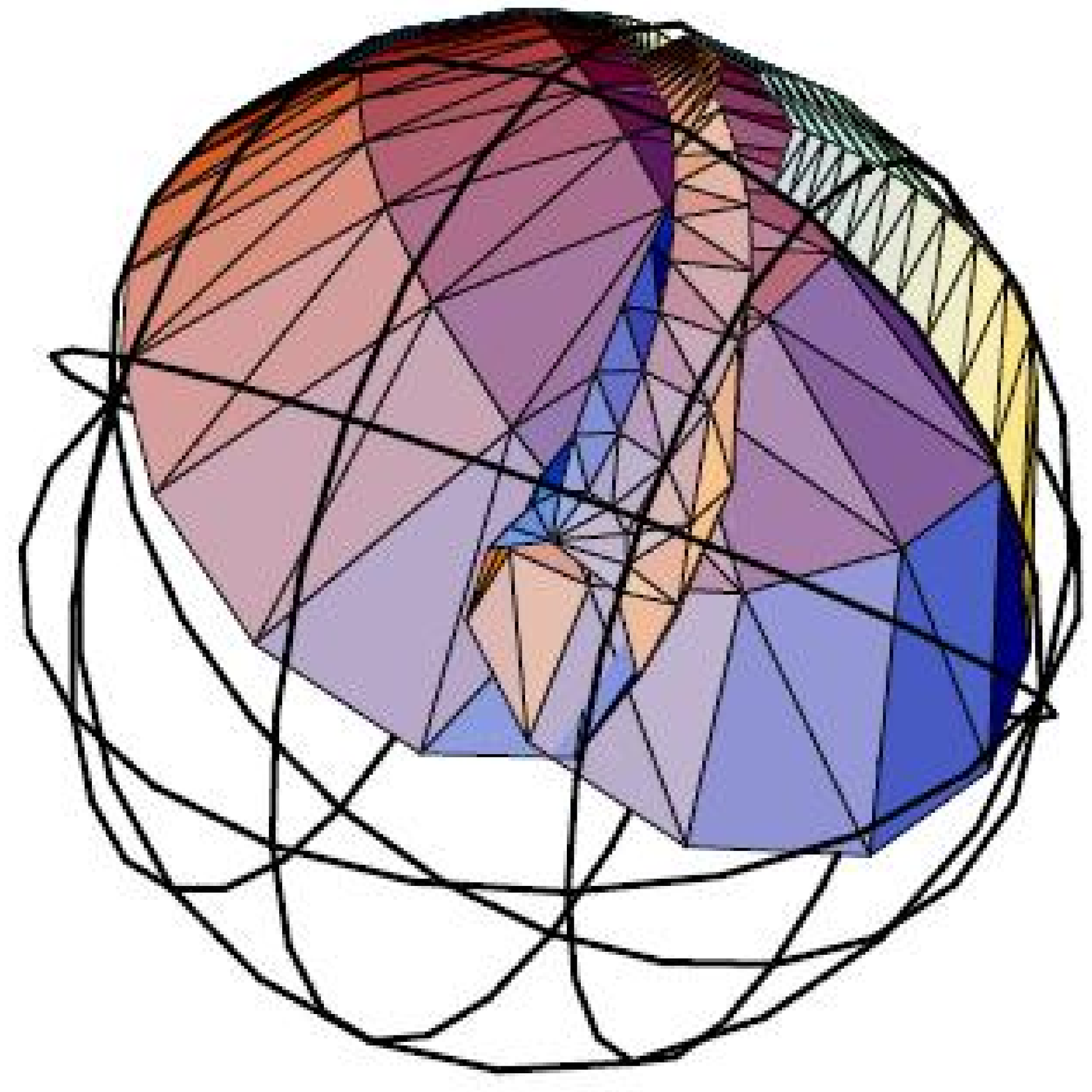}
        \hfill
\caption{A minimal Enneper surface in $\bfR^3$, and half of an Enneper 
cousin in the Poincare model 
for $\bfH^3$.  The entire Enneper cousin consists of the piece 
above and its reflection across the plane containing the planar 
geodesic boundary.}
\end{figure}

\begin{figure}
        \hspace{1.51in}
        \epsfxsize=1.8in
        \epsffile{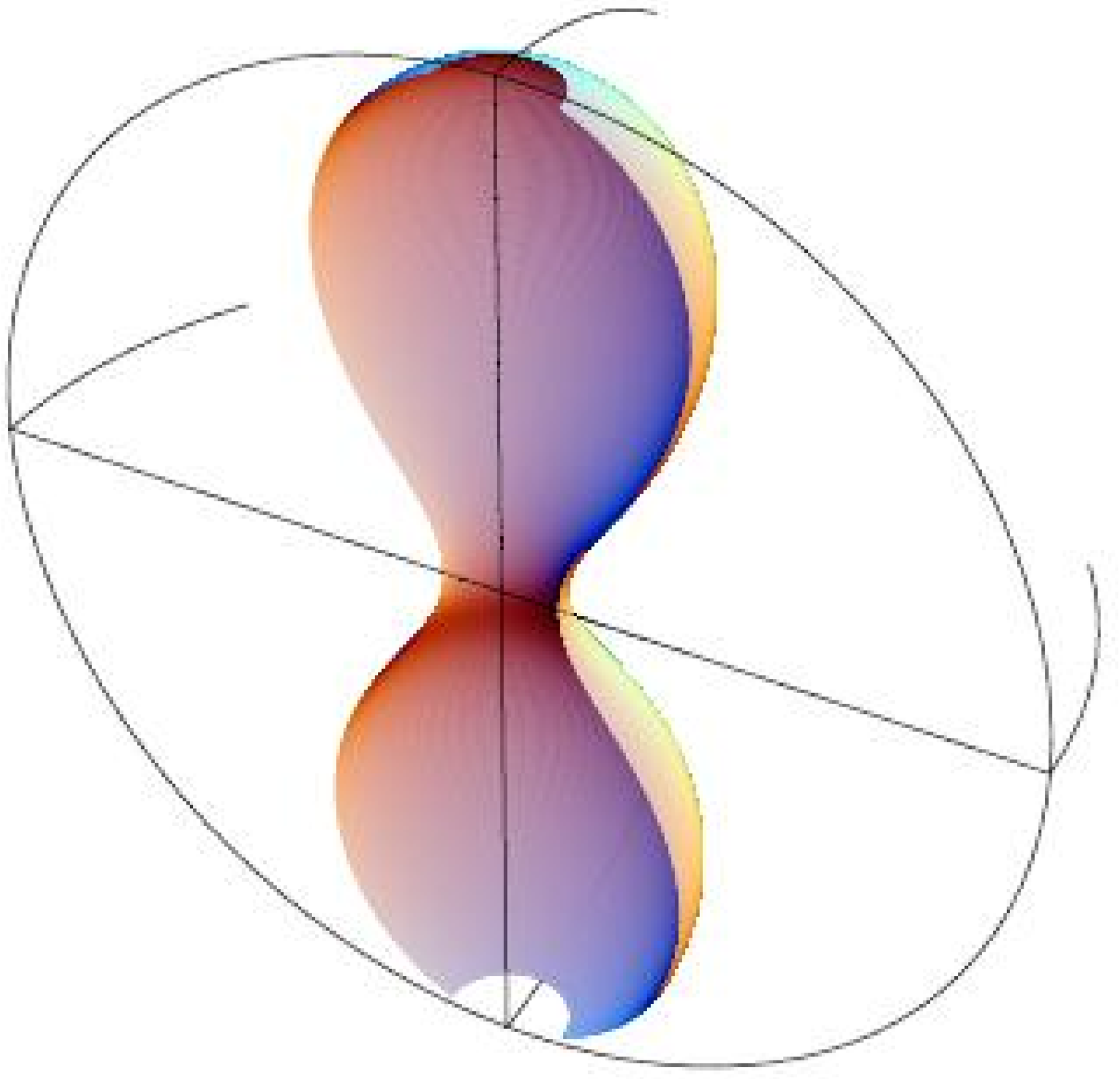}
        \hspace{0.01in}
        \epsfxsize=1.8in
        \epsffile{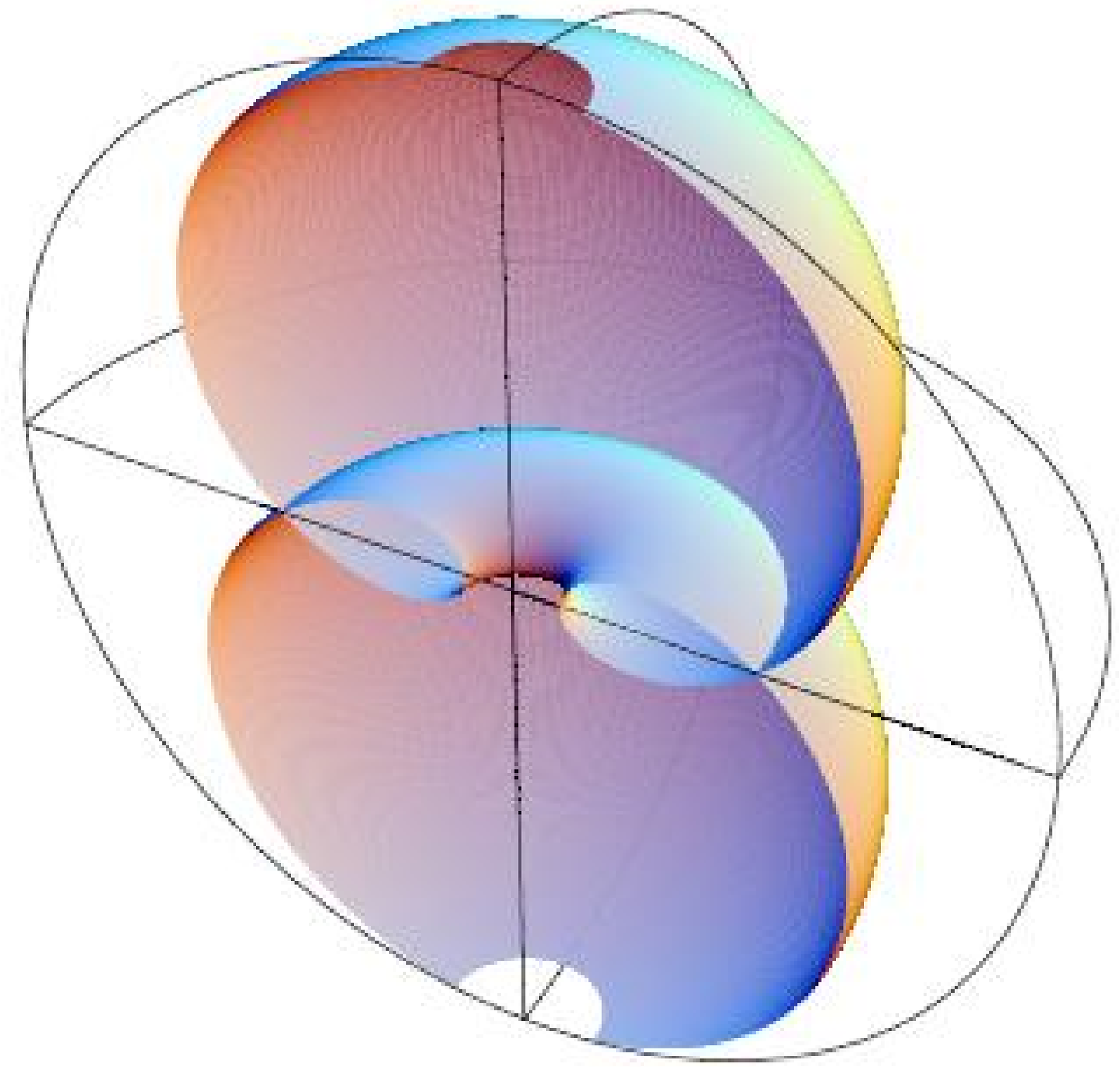}
        \hfill
\caption{Two different genus 0 catenoid cousins in 
the Poincare model for $\bfH^3$.
The surface on the left is embedded, and the surface on the right is 
not embedded.  }
%We draw only the central cross-section in the nonembedded 
%case, as this makes a more understandable picture than a picture of the 
%full surface: note that the curves in this cross-section are profile curves 
%for this surface of revolution.}
\end{figure}

We now describe some simple examples:
\begin{itemize}
\item The horosphere is a CMC 1 surfaces in $\bfH^3$.  It has 
Weierstrass data $\Sigma \setminus \{p_j \} = \bfC$, 
$g = 1$, $f = 1$.
\item The Enneper cousins (\cite{RUY}) 
(see Figure 2) have Weierstrass data 
$\Sigma \setminus \{p_j \} = \bfC$, $g = z$, $f = \lambda \in \bfR$.
\item The catenoid cousins (\cite{B}, \cite{UY1}) 
(see Figure 3) have Weierstrass data 
$\Sigma \setminus \{p_j \} = \bfC \setminus \{ 0 \}$, $g = z$, 
$f = \frac{\lambda}{z^2} \in \bfR$.  These surfaces are either 
embedded or not embedded, depending on the value of $\lambda$.  
\end{itemize}

We now state some known facts, which when taken together, further 
show just how closely related 
CMC 1 surfaces in $\bfH^3$ are to minimal surfaces 
in $\bfR^3$.

\begin{itemize}
\item It was shown in \cite{UY2} that if $f$ and $g$ and 
$\Sigma \setminus \{p_j\}$ are fixed, 
then as $c \to 0$, the CMC $c$ surfaces $\Phi$ in $\bfH^3(-c^2)$ 
converge locally 
to a minimal surface in $\bfR^3$.  This can be intuited 
from the fact that $G \to g$ as $c \to 0$ (which 
follows directly from equation \ref{eq:wode}), and hence 
the above first and second fundamental forms for the CMC $c$ 
surfaces $\Phi$ converge to the fundamental forms for a 
minimal surface as $c \to 0$ (up to a sign change in $II$ -- a 
change of orientation).  The resulting minimal surface does not 
necessarily have the same global topology as the CMC $c$ surfaces, 
and it may be periodic.  
\item Consider the Poincare model for $\bfH^3(-c^2)$ for $c \approx 0$.  
It is a round ball in $\bfR^3$ centered at the origin 
with Euclidean radius $\frac{1}{|c|}$ endowed with a complete 
radially-symmetric metric 
$ds_c^2 = \frac{4 \sum dx_i^2}{ ( 1 - c^2 \sum x_i^2 )^2 }$ 
of constant sectional curvature $-c^2$.  Contracting 
this model by a factor of $|c|$, we obtain a map to the Poincare model 
for $\bfH^3$.  Under this mapping, CMC $c$ 
surfaces are mapped to
CMC $1$ surfaces.  Thus the problem of existence of 
CMC $c$ surfaces in $\bfH^3(-c^2)$ 
for $c \approx 0$ is equivalent to the problem of existence of 
CMC $1$ surfaces in $\bfH^3$.  
\item It was shown in \cite{RUY} that a finite-total-curvature minimal
surface in $\bfR^3$ satisfying certain conditions (these conditions 
are fairly general and include most known examples) 
can be deformed into a CMC $c$ surface in 
$\bfH^3(-c^2)$ for $c \approx 0$, 
so that $\Sigma$, $f$, and $g$ are the same, up to 
a slight adjustment of the real parameters that are used to solve the 
period problems.  By the previous item, these surfaces are equivalent 
to CMC 1 surfaces in $\bfH^3$, thus we have a 1-parameter family of 
CMC 1 surfaces in $\bfH^3$ with parameter $c$.  
The deformed surfaces might not have finite total 
curvature, but they will be of the same topological type as the minimal 
surface, and they will have the same reflectional symmetries as the 
minimal surface.  (See Section 4.)  
\end{itemize} 

Regarding the last item above, Theorem 1.2 shows that the converse of 
the \cite{RUY} result does not hold.  

\section{Embedded Case}

The proof of Theorem 1.1 uses the maximum principle and
Alexandrov reflection.
Before stating the maximum principle, we define some terms.
If $\Sigma_1$ and $\Sigma_2$ are two smooth oriented complete
hypersurfaces of $\bfH^n$ which are tangent at a point $p$ and have
the same oriented normal at $p$, we say that $p$ is a {\em point of
  common
tangency} for $\Sigma_1$ and $\Sigma_2$.  Let the common
tangent geodesic hyperplane $\cal P$ through $p$
have the same orientation as $\Sigma_1$ and
$\Sigma_2$ at $p$.  Then, near $p$,
expressing $\Sigma_1$ and $\Sigma_2$
as graphs $g_1(x)$ and $g_2(x)$ over points $x
\in {\cal P}$ (the term {\em graph} in this context is defined in
\cite{CL}), we say that $\Sigma_1$ {\em lies above}
$\Sigma_2$ near $p$ if $g_1 \geq g_2$.

\begin{proposition}
(Maximum Principle)
Suppose that $\Sigma_1$ and $\Sigma_2$ are closed oriented 
hypersurfaces in $\bfH^{n}$ with the same 
constant mean curvature $c$ and the same smooth boundary 
$\partial \Sigma_1 = \partial \Sigma_2$.  
Suppose that $\Sigma_1$ and $\Sigma_2$ have a point $p$ of
common tangency, and that $\Sigma_1$ lies above $\Sigma_2$ near $p$.  
(The point $p$ can be either an interior point of 
both $\Sigma_1$ and $\Sigma_2$ or a boundary point of both 
$\Sigma_1$ and $\Sigma_2$.)  Then $\Sigma_1 = \Sigma_2$.
\end{proposition}

The above proposition is well known, and proofs can be found in 
\cite{KKMS}, \cite{CL}, and references therein.

In the Poincare model for $\bfH^3$ ($B^3 = \{x \in \bfR^3, |x| < 1\}$ 
with the metric $ds^2 = \frac{4|dx|^2}{(1 - |x|^2)^2}$), 
the totally geodesic planes are the intersections of $B^3$ with
spheres and planes in $\bfR^3$ which meet $\partial B^3$ orthogonally.
We shall use the Poincare model and these totally geodesic planes 
in the proof of Theorem~1.1, which we now give.  

\begin{proof}
We consider a complete properly embedded CMC surface $M$ in $\bfH^3$.  
First we suppose that $M$ has asymptotic boundary consisting of exactly 
two points.  Applying an isometry of 
$\bfH^{3}$ if necessary, we may assume that these two asymptotic 
points are at the north and south poles $(0,0,\pm 1)$.  

Let $\vec{v}$ be a horizontal unit vector in $\bfR^3$.  For $t \in
(-1,1)$, let $P_t$ be the totally geodesic plane containing the point 
$t \vec{v}$ and perpendicular to the line through $\vec{v}$.  The
plane $P_t$ separates $\bfH^3$ into two regions: let $A_t$ be the
region containing the points $s \vec{v}, s \in (-1,t)$, and let
$B_t$ be the region containing the points $s \vec{v}, s \in (t,1)$.  
Let $(M \cap A_t)^\prime$ be the isometric reflection of $M \cap A_t$
across $P_t$.  

Let $t_{0}$ be the largest value $t_0$ such that for all $t$ less than
$t_0$, Int($(M \cap A_t)^\prime$) and Int($M \cap B_t$) are disjoint.  
When $t$ is close to -1 or 1, $P_t \cap M$ is empty, so it follows 
that such a $t_{0}$ exists and that $t_{0} \in (-1,1)$.  It then 
follows (since $M$ is properly embedded) 
that there exists a finite point of common tangency between $(M \cap 
A_{t_0})^\prime$ and $M \cap B_{t_0}$, and that one surface lies above 
the other in a neighborhood of this point of common tangency.  The
maximum principle implies that $M \cap B_{t_0} = (M \cap
A_{t_0})^\prime$.  (This is the Alexandrov reflection principle.)  
Since $M$ has only two ends at the north and south
poles, it must be that $t_0 = 0$.  Since $\vec{v}$ was 
an arbitrary horizontal vector, it
follows that $M$ is symmetric with respect to any 
geodesic plane through the north and south poles.  Thus it is a
surface of revolution.  

If the surface $M$ has no ends or only one point in its asymptotic 
boundary, one can similarly 
conclude that $M$ is a surface of revolution (spheres and horospheres).  
\end{proof}

\section{Immersed Counterexample}

In sections 5 and 6, 
we give a rigorous proof of existence of the genus 1 catenoid cousin.
However, since the proof itself does not enlighten the reader as to why
it should exist, we give a motivation in this section
for why we should expect this surface to exist.

As noted in section 2, it was shown 
in \cite{RUY} that for any complete finite total curvature
minimal surface in $\bfR^3$ which satisfies certain conditions, 
there exists a corresponding 1-parameter family of CMC 1 surfaces 
in $\bfH^3$.  So, in this sense, the set of CMC 1 surfaces in $\bfH^3$ is
a larger set than the set of minimal surfaces in $\bfR^3$.  We now 
briefly sketch 
the ideas behind the result in \cite{RUY}.  We do not describe the method
in detail, as the reader can refer to \cite{RUY}.  

We start with a given minimal surface in $\bfR^3$, and thus have a 
Riemann surface $\Sigma$ and 
meromorphic functions $f$ and $g$ given to us by the 
Weierstrass representation for minimal surfaces in $\bfR^3$ (Lemma 
2.1).  We can 
use this same $\Sigma$ and $f$ and $g$ in the 
hyperbolic Weierstrass representation (Lemma 
2.2) to produce a CMC $c$ surface in 
$\bfH^3(-c^2)$ for each real number $c$.  

As $c \rightarrow 0$, the Poincare model for $\bfH^3(-c^2)$
(a ball in $\bfR^3$ with Euclidean radius $\frac{1}{|c|}$) converges to 
Euclidean space $\bfR^3$, and these CMC $c$ surfaces in
$\bfH^3(-c^2)$ converge to the given minimal surface (in the sense of 
$C^\infty$ uniform convergence on compact sets).  Thus for $c$ close to zero, 
we can think of compact regions of the CMC $c$ surfaces in
$\bfH^3(-c^2)$ as small deformations of compact regions of 
the given minimal surface in $\bfR^3$.  

If the given minimal surface in $\bfR^3$ is not simply-connected, then
there is a question about whether the deformed CMC $c$ surfaces in 
$\bfH^{3}(-c^{2})$ are well-defined.  This is the period
problem.  (The period problem being solvable essentially means that a certain 
set of equations Per$_j(\lambda_i) = 0$ can be solved with respect to certain 
parameters $\lambda_i$ of the surface.  This will be explained in detail in 
terms of an $SU(2)$ condition in the next section.)  
The minimal surface is assumed to have a "nondegeneracy" 
property, as defined in \cite{RUY}.  
Since the period problem is nondegenerate and solvable on the 
minimal surface, and since the period problem changes 
continuously with respect to $c$, 
the period problem can still be solved 
when $c$ is sufficiently close to 0.  Thus, for $c$ sufficiently close 
to 0, the CMC $c$ surfaces in
$\bfH^3(-c^2)$ are well-defined.

Dilating the Poincare model for $\bfH^3(-c^2)$ by 
a factor of $|c|$ (as described in section 2), we produce 
a one-parameter family of CMC 1 surfaces in $\bfH^{3}$, with 
parameter $c$.  
This is the method used in \cite{RUY} to create well-defined 
non-simply-connected CMC 1 surfaces in $\bfH^3$ 
from non-simply-connected minimal surfaces in $\bfR^3$.

As an example of this, consider the minimal 
genus 1 trinoid in $\bfR^3$.  As discussed in 
\cite{BR}, there is a single real parameter $\lambda$ 
in the Weierstrass data that 
can be adjusted to solve the period problem.  The period problem is 
represented by a map $\lambda \in \bfR \rightarrow \mbox{Per}(\lambda) 
\in \bfR$, and to solve the period problem we must show that there exists 
a value of $\lambda$ so that Per($\lambda$) = 0.  We 
note that the function Per($\lambda$) changes 
continuously in $c$.  Since the period problem for the minimal 
genus 1 trinoid in $\bfR^{3}$ (when $c$ is 0) is solvable and 
nondegenerate, this implies that there exists an interval
$(a,b) \in \bfR$ so that the image of this interval under the map Per 
contains an interval about 0.  By continuity, if we perturb $c$ slightly 
away from 0, it still holds that $0 \in$ Per($(a,b)$).  Thus, for $c$ 
sufficiently close to zero, there exists a CMC $c$ genus 1 trinoid cousin 
in $\bfH^3(-c^2)$.  Then, by dilating the Poincare model, we produce a 
CMC 1 genus 1 trinoid cousin in $\bfH^3(-1)$  (see Figure 4).

\begin{figure}
        \hspace{1.51in}
        \epsfxsize=1.8in
        \epsffile{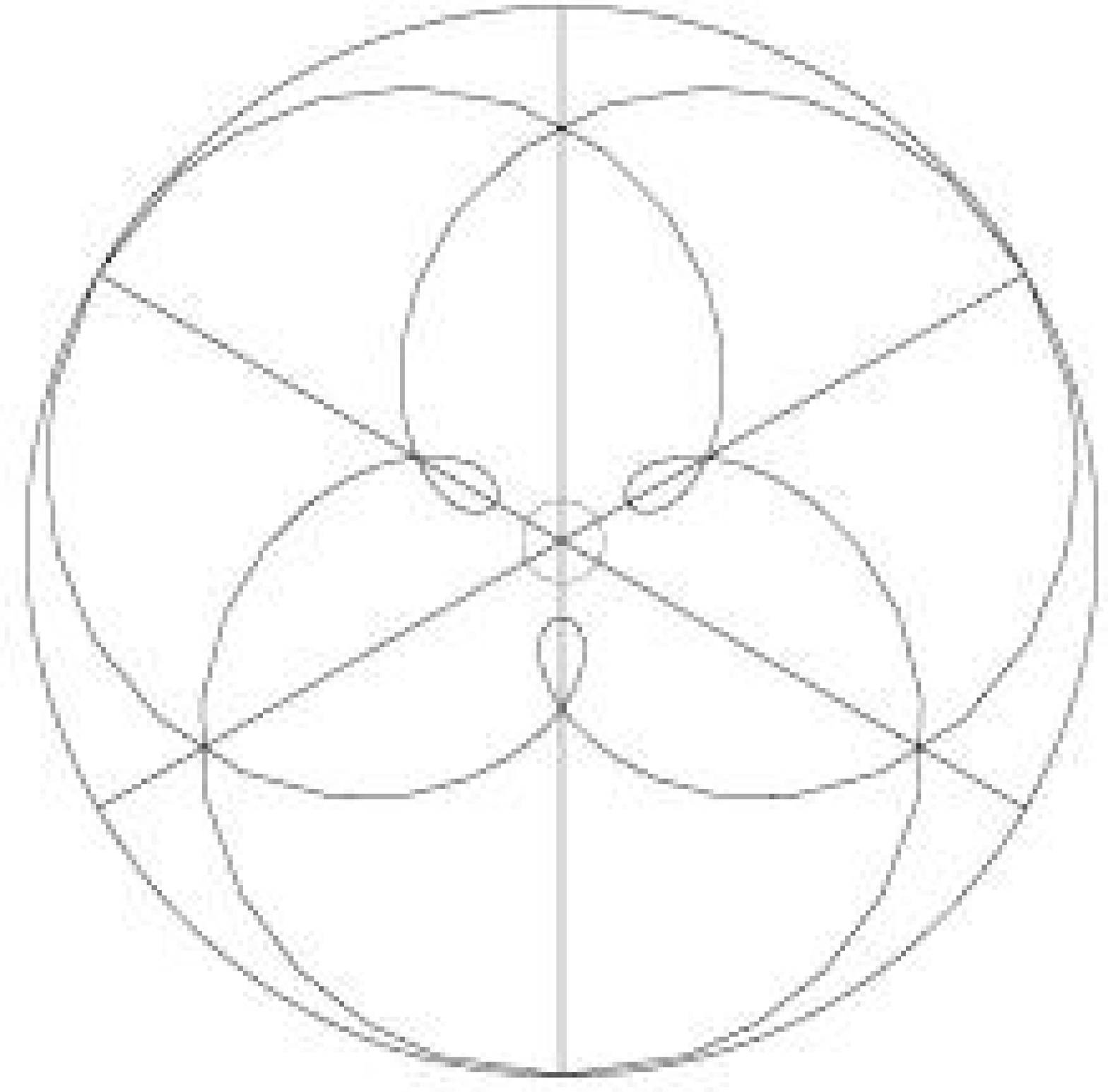}
        \hspace{0.01in}
        \epsfxsize=1.8in
        \epsffile{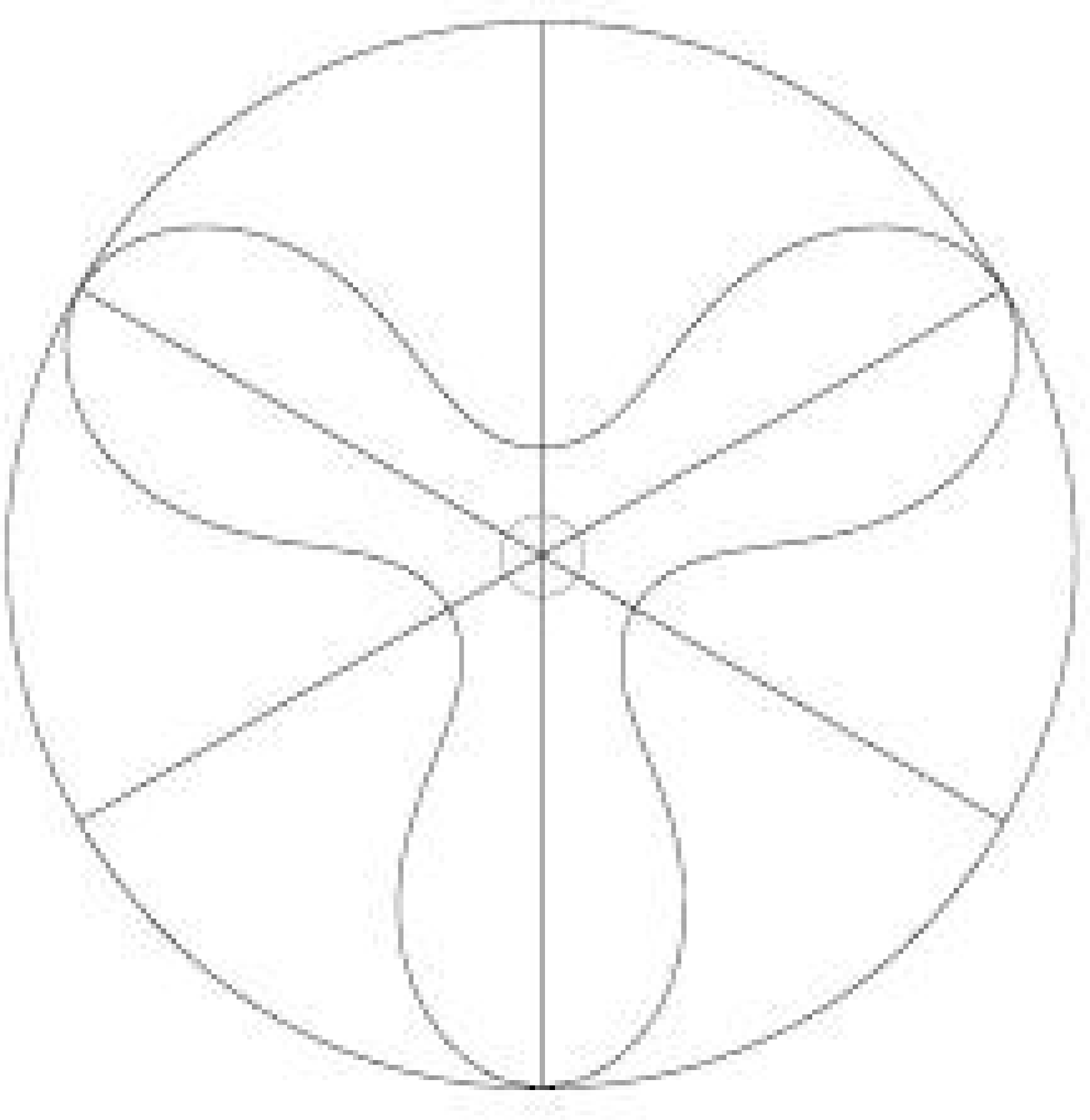}
        \hfill
\caption{These are slices of genus 1 trinoid cousins along a plane of 
reflective symmetry.  The lefthand picture is an embedded genus 1 
trinoid cousin, produced by using a negative value for $c$.
The righthand picture is an immersed genus 1 trinoid cousin, produced by 
using a positive value for $c$.}
\end{figure}

Now let us consider Weierstrass data that would produce a minimal 
genus 1 catenoid is $\bfR^3$.  Again there is a single 
real parameter $\lambda$ in the Weierstrass data that can be adjusted, 
and again the period problem is 
represented by a map $\lambda \in \bfR \rightarrow \mbox{Per}(\lambda) 
\in \bfR$, and to solve the period problem we must again 
show that there exists 
a value of $\lambda$ so that Per($\lambda$) = 0.  
The Weierstrass data is described in 
the next section.  In this case the period problem cannot be 
solved, since the function Per is always positive, but the function 
Per can get 
arbitrarily close to 0.  So we have that an interval $(0,\epsilon)$ is 
contained in the image of Per, for some $\epsilon > 0$.  If we perturb $c$ 
slightly, we expect that 
the image interval $(0,\epsilon)$ gets perturbed continuously.  The authors 
have found by numerical experiment that when $c$ becomes slightly 
negative, the image interval $(0,\epsilon)$ changes to an image interval of 
the form $(\epsilon_1,\epsilon_2)$, where $0 < \epsilon_1 < \epsilon_2$.  
Thus as $c$ becomes negative, the lower endpoint of the 
image interval moves in a positive 
direction.  If this behavior were nondegenerate at $c=0$, then one would 
expect that as $c$ is perturbed slightly in a positive direction, the image 
interval would become of the form 
$(\epsilon_1,\epsilon_2)$, where $\epsilon_1 < 0 < \epsilon_2$.  We have 
found by numerical experiment that this is indeed the case.  Thus for 
slightly positive values of $c$, we can adjust a real parameter in the 
Weierstrass data so that the period function Per becomes zero.  Then, 
after dilating the Poincare model, we have 
existence of a CMC 1 genus 1 catenoid cousin in $\bfH^3(-1)$.

The behavior of the genus 1 catenoid as $c$ is perturbed is similar to the 
behavior of the genus 1 trinoid as $c$ is perturbed.  As $c$ becomes 
negative, the genus 1 trinoid cousin eventually becomes embedded.  
If the period 
problem could be solved for the genus 1 catenoid cousin when $c<0$, the 
resulting surface would be embedded, but as we know by Theorem 1.1, such a 
surface cannot exist.  But when solving the period problem for $c>0$, the 
genus 1 catenoid cousin is not embedded, just in the same way that the 
genus 1 trinoid cousin is not embedded for $c>0$.  
So this numerical experiment is 
consistent with Theorem~1.1, and furthermore shows that Theorem~1.1 
holds only for embedded surfaces.

In the next sections, we will prove Theorem 1.2.  
First we describe the Weierstrass 
representation and the period problem.  Then we give the proof.  
Initially the
period problem is a three dimensional problem.  We 
reduce the period problem by algebraic arguments
to a one dimensional problem, then we show by a numerical calculation
and the intermediate value theorem that this one dimensional 
period problem can be
solved.  The remainder of the 
proof is essentially a mathematically rigorous verification 
that the numerical experiment we conducted is correct.  
We must give mathematically rigorous bounds for both computer 
round-off error and for error introduced by discretizing the problem.
We will show the errors are sufficiently small to ensure existence of a 
solution to the period problem.

\section{Period Problem for the Genus 1 Catenoid Cousin}

Consider the Riemann surface 
\[{\cal M}_{a}^2 = \{(z,w) \in \bfC \times (\bfC \cup \{ \infty \}) : \; 
(z-1)(z+a)w^2 = (z+1)(z-a)\}\]
for $a > 1$.  Thus ${\cal M}_{a}^2$ is a twice punctured torus.  
Let $g = w$ and let $f = \frac{c}{w}$, for $c > 0$.  (This $c$ is 
the same as the $c$ described in the previous section, and $a$ is the 
same as $\lambda$ in the previous section.)  
The Riemannn surface ${\cal M}_{a}^2$ and 
meromorphic functions $g$ and $f$ are the 
Weierstrass data for a genus 1 catenoid.  Let $F(z,w) \in SL(2,\bfC)$
satisfy Bryant's equation 
\[ F^{-1}dF = \left( \begin{array}{cc} 
g & -g^2 \\ 
1 & -g 
\end{array} \right) f dz \]
with initial condition $F(z=0,w=1)$ = identity.  
Hence $\Phi = F^{-1} \overline{F^{-1}}^t$ is a CMC
1 surface in the Hermitean model for $\bfH^3$, and this surface is 
defined on the universal cover of 
${\cal M}_{a}^2$.  (Representing $\Phi$ in this way, we have already 
done the dilation of hyperbolic space which produces a CMC 1 surface 
in $\bfH^3(-1)$ from a CMC $c$ surface in $\bfH^3(-c^2)$.)  

However, we don't yet know that $\Phi$ is
well-defined on ${\cal M}_{a}^2$ itself 
(which must be the case if $\Phi$ has finite total curvature 
$\int_{{\cal M}_{a}^2} -K dA < + \infty$).  
To have this, we need that $F$ 
satisfies the $SU(2)$ condition, which we now describe.  Suppose that 
$\gamma$ is a loop in ${\cal M}_{a}^2$ with base point $p \in {\cal M}_{a}^2$.  
Suppose that the value of $F$ at $p$ is $F(p)$.  Starting with the initial 
condition $F(p)$ and evaluating $F$ along $\gamma$ using Bryant's equation 
above, we return to the base point $p$ with a new value $\breve{F}(p)$ 
for $F$ at $p$.  If the loop $\gamma$ is nontrivial, we can expect that 
$\breve{F}(p) \neq F(p)$.  However, since both $\breve{F}(p)$ and $F(p)$ 
are in $SL(2,\bfC)$, there exists a matrix $P \in SL(2,\bfC)$ such that 
$\breve{F}(p) = P \cdot F(p)$.  If $P \in SU(2)$, then it follows that 
$\breve{F}^{-1} \overline{\breve{F}^{-1}}^t = F^{-1} 
\overline{F^{-1}}^t$.  
Thus if $P \in SU(2)$ for any loop $\gamma$, then 
$\Phi$ is well-defined on ${\cal M}_{a}^2$ itself.  We say that the 
$SU(2)$ condition is satisfied on $\gamma$ if $P \in SU(2)$.

\begin{figure}
        \hspace{1.01in}
        \epsfxsize=3.4in
        \epsffile{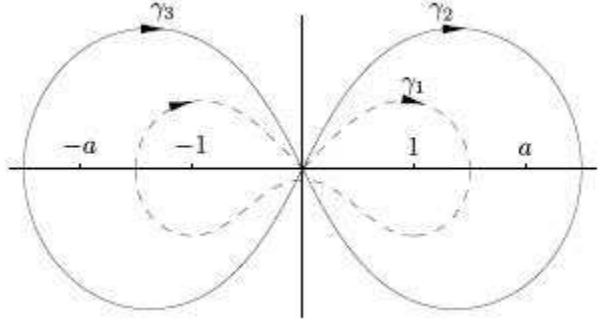}
        \hfill
\caption{The curves $\gamma_1$, $\gamma_2$, and $\gamma_3$ generate 
the fundamental group of ${\cal M}_{a}^2$.}
\end{figure}

It is enough to check the $SU(2)$ condition on the following three
loops, as they generate the fundamental group of ${\cal M}_{a}^2$ (see 
Figure 5):  
\begin{itemize}
\item Let 
$\gamma_1 \subset {\cal M}_{a}^2$ be a curve starting at $(0,1) \in 
{\cal M}_{a}^2$
whose first portion has $z$ coordinate in the first quadrant of the $z$ 
plane and ends at a point $(z,w)$ where $z \in \bfR$ and $1<z<a$, and whose 
second portion starts at $(z,w)$ and ends at $(0,-1)$ and has $z$ coordinate in 
the second quadrant, and whose 
third portion starts at $(0,-1)$ and ends at 
$(-z,\frac{1}{\bar{w}})$ and has $z$ coordinate in 
the third quadrant, and whose fourth and 
last portion starts at $(-z,\frac{1}{\bar{w}})$ and returns to the base 
point $(0,1)$ and has $z$ coordinate in the fourth quadrant.  
\item Let $\gamma_2 \subset {\cal M}_{a}^2$ be a curve starting at $(0,1)$
whose first portion has $z$ coordinate in the first quadrant and ends at a 
point $(z,w)$ where $z \in \bfR$ and $z>a$, and whose 
second and last portion starts at $(z,w)$ and returns to $(0,1)$ and has $z$ 
coordinate in the second quadrant.  
\item Let $\gamma_3 \subset {\cal M}_{a}^2$ be a curve starting at $(0,1)$
whose first portion has $z$ coordinate in the third quadrant and ends at a 
point $(z,w)$ where $z \in \bfR$ and $z<-a$, and whose 
second and last portion starts at $(z,w)$ and returns to $(0,1)$ and has $z$ 
coordinate in the fourth quadrant.  
\end{itemize}

Consider the following symmetries from ${\cal M}_{a}^2$ to 
${\cal M}_{a}^2$: $\phi_1(z,w) =
(\bar{z},\bar{w})$, $\phi_2(z,w) = (-z,\frac{1}{w})$,
$\phi_3(z,w) = (-\bar{z},\frac{1}{\bar{w}})$.  
If $(z(t),w(t)), t \in [0,1]$ is a curve in ${\cal M}_{a}^2$ 
that begins at $(0,1)$ when $t=0$ and ends at some point $(z,w)$ when $t=1$, 
then we can consider how
$F$ changes along $(z(t),w(t))$.  At the beginning point of
$(z(t),w(t))$ let $F(0,1)$ be
the identity, and denote the value of $F$ at the ending point of
$(z(t),w(t))$ by $F(z,w)$.  Then if we consider the curve 
$\phi_i(z(t),w(t))$, $F$ is the identity at the beginning of
this curve 
as well, and we denote the value of $F$ at the end of this curve by
$F(\phi_i(z,w))$.  The following lemma gives the relationships between
$F(z,w)$ and $F(\phi_i(z,w))$.  

\begin{lemma}
\newcounter{num}
%\begin{list}%
%{\arabic{num})}{\usecounter{num}\setlength{\rightmargin}{\leftmargin}}

If $F(z,w) = \left(\begin{array}{cc} A & B \\
                   C & D
                \end{array}\right)$, 
                then 
$F(\phi_1(z,w)) = \left(\begin{array}{cc} \bar{A} & \bar{B} \\
                   \bar{C} & \bar{D}
                \end{array}\right)$
and 
$F(\phi_2(z,w)) = \left(\begin{array}{cc} D & C \\
                   B & A
                \end{array}\right)$
and 
$F(\phi_3(z,w)) = \left(\begin{array}{cc} \bar{D} & \bar{C} \\
                   \bar{B} & \bar{A}
                \end{array}\right)$.
%\end{list}
\end{lemma}

\begin{proof}
Suppose $F(0,1)$ is 
the identity and $F = \left(\begin{array}{cc} A & B \\
                   C & D
                \end{array}\right)$ is a solution to the equation 
\[  \left(\begin{array}{cc}
                            dA & dB \\
                    dC & dD
                \end{array}\right) = \left(\begin{array}{cc}
                            A & B \\
                    C & D
                \end{array}\right)\left(\begin{array}{cc}
                            1 & -g \\
                    \frac{1}{g} & -1
                \end{array}\right) cdz \;  \]
on $(z(t),w(t))$.  Equivalently, 
\[  \left(\begin{array}{cc}
                                    d\bar{A} & 
                                    d\bar{B} \\
                    d\bar{C} & d\bar{D}
                \end{array}\right) = \left(\begin{array}{cc}
                            \bar{A} & \bar{B} \\
                    \bar{C} & \bar{D}
                \end{array}\right)\left(\begin{array}{cc}
                            1 & -\bar{g} \\
                    \frac{1}{\bar{g}} & -1
                \end{array}\right) cd\bar{z} \;  \] on $(z(t),w(t))$.  
Since when $(z,w) \rightarrow \phi_1(z,w)$, we have that 
$z \rightarrow \bar{z}$ and $g \rightarrow \bar{g}$, we conclude 
that $\left(\begin{array}{cc}
                            \bar{A} & \bar{B} \\
                    \bar{C} & \bar{D}
                \end{array}\right)$ is a solution on 
$\phi_1(z(t),w(t))$.  Since the initial condition $F=$identity is left 
unchanged by conjugation, we conclude the first part of the lemma.
The above equation could also be equivalently written as 
\[  \left(\begin{array}{cc}
                            dD & dC \\
                    dB & dA
                \end{array}\right) = \left(\begin{array}{cc}
                            D & C \\
                    B & A
                \end{array}\right)\left(\begin{array}{cc}
                            1 & -\frac{1}{g} \\
                   g & -1
                \end{array}\right) cd(-z) \; . \]
Since when $(z,w) \rightarrow \phi_2(z,w)$, 
we have that 
$z \rightarrow -z$ and $g \rightarrow \frac{1}{g}$,
% we conclude 
%that $\left(\begin{array}{cc}
%                            D & C \\
%                    B & A
%                \end{array}\right)$ is a solution on 
%$\phi_2(z(t),w(t))$.  Since the initial condition $F=$identity is left 
%unchanged by the map $\left(\begin{array}{cc}
%                            A & B \\
%                  C & D
%                \end{array}\right) \rightarrow 
%                \left(\begin{array}{cc}
%                            D & C \\
%                    B & A
%                \end{array}\right)$, 
                we can conclude the 
                second part of the lemma in the same way.
Since $\phi_3=\phi_2 \circ \phi_1$, the first two parts of the lemma imply 
the final part.
\end{proof}

In the next lemma, we consider the map 
$\phi_4(z,w) = (\bar{z},-\bar{w})$.  This map is different 
from the other three maps $\phi_i(z,w)$, $i=1,2,3$ 
in that $(0,1)$ is not in the fixed point set of $\phi_4$.  
Thus when $(z(t),w(t))$ is a curve that
begins at $(0,1)$, $\phi_4(z(t),w(t))$ is a curve that
begins at $(0,-1)$, not $(0,1)$.  

\begin{lemma}
Suppose that $(z(t),w(t)) \subset {\cal M}^2_a$ is a curve that starts at
  $(0,1)$ and ends at a point $(z,w)$ 
  such that $z \in \bfR$ and $1<z<a$.  Evaluating Bryant's equation 
  along $(z(t),w(t))$ with initial condition $F(0,1)=$identity, 
  we denote the value of $F$ at the endpoint $(z,w)$ by 
  $F(z,w)=\left(\begin{array}{cc} A & B \\
                   C & D
                \end{array}\right)$.  Then 
  $\phi_4(z(t),w(t))$ starts at $(0,-1)$ and ends at the same endpoint 
  $(z,w)$.  If we evaluate Bryant's equation along $\phi_4(z(t),w(t))$ 
  with initial condition $F(0,-1)=$identity, then the value of $F$ at 
  the endpoint $(z,w)$ of $\phi_4(z(t),w(t))$ is 
  $F(\phi_4(z,w))=\left(\begin{array}{cc} \bar{A} & -\bar{B} \\
                   -\bar{C} & \bar{D}
                \end{array}\right)$.
\end{lemma}

\begin{proof}
Bryant's equation can be equivalently written as 
\[  \left(\begin{array}{cc}
                                    d\bar{A} & 
                                    -d\bar{B} \\
                    -d\bar{C} & d\bar{D}
                \end{array}\right) = \left(\begin{array}{cc}
                            \bar{A} & -\bar{B} \\
                    -\bar{C} & \bar{D}
                \end{array}\right)\left(\begin{array}{cc}
                            1 & \bar{g} \\
                    -\frac{1}{\bar{g}} & -1
                \end{array}\right) cd\bar{z} \; . \] 
The result follows just as in the previous proof.  
%Since when $(z,w) \rightarrow \phi_4(z,w)$, we have that 
%$z \rightarrow \bar{z}$ and $g \rightarrow -\bar{g}$, the above equation 
%implies that $\left(\begin{array}{cc}
%                            \bar{A} & -\bar{B} \\
%                    -\bar{C} & \bar{D}
%                \end{array}\right)$ is a solution on 
%$(\phi_4(z(t),w(t)))$.  Since the initial condition $F=$identity is left 
%unchanged by the map $\left(\begin{array}{cc}
%                            A & B \\
%                    C & D
%                \end{array}\right) \rightarrow 
%                \left(\begin{array}{cc} \bar{A} & -\bar{B} \\
%                   -\bar{C} & \bar{D}
%                \end{array}\right)$, we conclude the lemma.
\end{proof}

Let $\alpha_1(t), t \in [0,1]$ be a curve starting at 
$(z,w) = (0,1)$ whose projection to the $z$-plane is an embedded curve
in the first quadrant, and whose endpoint has a $z$ coordinate that is real 
and larger than 1 and less than $a$.
Let $\alpha_2(t), t \in [0,1]$ be a curve starting at 
$(z,w) = (0,1)$ whose projection to the $z$-plane is an embedded curve
in the first quadrant, and whose endpoint has a $z$ coordinate that is real 
and larger than $a$.  With $F$=identity at $(z,w) = (0,1)$, we solve 
Bryant's equation along these two paths to find that 
\[ F(\alpha_1(1)) = \left(\begin{array}{cc} A_1 & B_1 \\
                   C_1 & D_1 
                \end{array}\right) \; , \;
F(\alpha_2(1)) = \left(\begin{array}{cc} A_2 & B_2 \\
                   C_2 & D_2
                \end{array}\right) \; . \]
Then traveling about the loop $\gamma_1$, it follows from Lemmas 5.1 and 
5.2 that $F$ changes from the
identity to the matrix
\[ \phi := \left(\begin{array}{cc} A_1 & B_1 \\
                   C_1  & D_1
                \end{array}\right)\left(\begin{array}{cc}  \bar{D}_1 & 
                \bar{B}_1 \\
                   \bar{C}_1 & \bar{A}_1
                \end{array}\right)\left(\begin{array}{cc} D_1 & -C_1 \\
                   -B_1  & A_1
                \end{array}\right)\left(\begin{array}{cc} \bar{A}_1 & -\bar{C}_1 \\
                   -\bar{B}_1 & \bar{D}_1
                \end{array}\right) \; . \]
And traveling about the loop $\gamma_2$, it follows from Lemma 5.1 that 
$F$ changes from the identity to the matrix
\[ \psi := \left(\begin{array}{cc} A_2 & B_2 \\
                    C_2 & D_2
                \end{array}\right)
\left(\begin{array}{cc} \bar{D}_2 & -\bar{B}_2 \\
                    -\bar{C}_2 & \bar{A}_2
                \end{array}\right) \; . \]
And traveling about $\gamma_3$, $F$ changes from the identity to the matrix
\[ \left(\begin{array}{cc} D_2 & C_2 \\
                    B_2 & A_2
                \end{array}\right)
\left(\begin{array}{cc} \bar{A}_2 & -\bar{C}_2 \\
                    -\bar{B}_2 & \bar{D}_2
                \end{array}\right) \; . \]
Changing the initial condition from $F(0,1) = $ identity to 
\[ F(0,1) = \left(\begin{array}{cc} \alpha & \beta \\
                   \beta & \alpha 
                \end{array}\right) \; , \]
where $\alpha, \beta \in \bfR$, $\alpha^2 - \beta^2 = 1$, we see that
solving the $SU(2)$ conditions on all three loops $\gamma_1$, $\gamma_2$, 
and $\gamma_3$ is equivalent to showing that 
\[ \left(\begin{array}{cc} \alpha & \beta \\
                    \beta & \alpha
                \end{array}\right)
\phi
\left(\begin{array}{cc} \alpha & -\beta \\
                    -\beta & \alpha
                \end{array}\right) \; \; , \; \;
                 \left(\begin{array}{cc} \alpha & \beta \\
                    \beta & \alpha
                \end{array}\right)
\psi
\left(\begin{array}{cc} \alpha & -\beta \\
                    -\beta & \alpha
                \end{array}\right) \]
are both in SU(2).  
%, where   
%\[ \phi=\left(\begin{array}{cc} 
%|A_1\bar{D}_1+B_1\bar{C}_1|^2-(A_1\bar{B}_1+\bar{A}_1B_1)^2 & 
%2(A_1\bar{D}_1+B_1\bar{C}_1)(\mbox{Re}(A_1\bar{B}_1)-\mbox{Re}(C_1\bar{D}_1)) \\
%2(\bar{A}_1D_1+\bar{B}_1C_1)(\mbox{Re}(C_1\bar{D}_1)-\mbox{Re}
%(A_1\bar{B}_1)) & 
%|A_1\bar{D}_1+B_1\bar{C}_1|^2-(C_1\bar{D}_1+\bar{C}_1D_1)^2
%                \end{array}\right) , \]\[\psi=
%\left(\begin{array}{cc} A_2\bar{D}_2-B_2\bar{C}_2 &  \bar{A}_2B_2-A_2\bar{B}_2 \\
%                   \bar{D}_2C_2-D_2\bar{C}_2 & \bar{A}_2D_2-\bar{B}_2C_2
%                \end{array}\right) \; \; . \]
We can choose an $\alpha$ and $\beta$ so that this holds precisely
when 
\[  f_1 := \frac{-2(\bar{A_1}D_1 + \bar{D}_1A_1 + 
\bar{C}_1B_1 + \bar{B}_1C_1)}{\bar{D}_1C_1 + \bar{C}_1D_1 + 
\bar{B}_1A_1 + \bar{A}_1B_1} = 
\frac{2(\bar{A}_2D_2 - \bar{D}_2A_2 + 
\bar{C}_2B_2 - \bar{B}_2C_2)}{\bar{D}_2C_2 - \bar{C}_2D_2 + 
\bar{B}_2A_2 - \bar{A}_2B_2}
 =: f_2\] and the absolute value of this 
number is greater than 2.  If this holds, we choose $\alpha$ and $\beta$ 
so that 
\[  f_1 = \frac{1+2\beta^2}{\beta\sqrt{1+\beta^2}} = f_2 \; \; ,\]
and then the $SU(2)$ conditions are satisfied.  

In order to prove Theorem 1.2, we need to show there exist values 
$c$ and $a$ so that $c>0$, $a>1$, $|f_1|=|f_2|>2$, and $f_1=f_2$.  In this 
next section we check that such values for $c$ and $a$ exist, by doing a 
mathematically rigorous analysis of the error bounds for our numerical 
approximations.  (See Figure~6.)

%Checking on a computer, as we adjust $a$, 
%by the intermediate value theorem there exists a value for $a$ so
%that the above problem is solved.  
%For example, when $k$ is fixed to be 0.02, we find that when a =
%1.193445, 
%$f_1 = f_2$ = $(2.866)^5$.  Choosing $\beta$ to be 0.00517174, we
%see that the above equation is satisfied.  Some other examples are
%$k=0.02$, a = 34.4355, $f_1 = f_2$ = $(2.182)^5$, $\beta=0.0202299$, and 
%$k=0.05$, a = 1.78, $f_1 = f_2$ = $(2.035)^5$, $\beta=0.0286889$, and 
%$k=0.05$, a = 10.482, $f_1 = f_2$ = $(1.8163)^5$, $\beta=0.0507852$, and 
%$k=0.06766$, a = 3.84, $f_1 = f_2$ = $(1.7555)^5$, $\beta=0.0603052$.  
%For the cases $k=0.07$ and $k=-0.1$, we do not get any 
%intersection points.  So it seems that there exists a pair of surfaces
%for each positive value of $k$ less that 0.06676, and that for some
%value approximately 0.06676, we have a unique surface.

\begin{figure}
        \hspace{1.01in}
        \epsfxsize=3.4in
        \epsffile{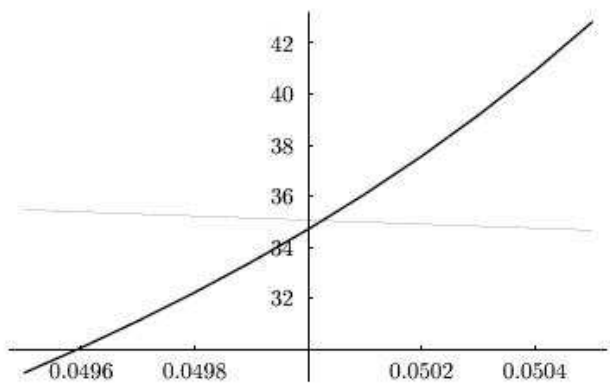}
        \hfill
\caption{The functions $f_1$ (the gray curve) and 
$f_2$ (the black curve) when $a=1.78$.  
The  horizontal axis represents $c$, and the 
vertical axis represents $f_1$ and $f_2$.  We can see that $f_1, 
f_2 > 2$ on the interval $c \in (0.0495,0.0505)$, and $f_1 = f_2$ at 
some value of $c$, and $a > 1$.}
\end{figure}

\section{Error Estimates}

Here we shall prove that for some given value of $a>1$ there exists a 
positive value for $c$ so that $f_1 = f_2 > 2$.  We 
do this by showing that for one particular value for $a$, there exists a
value of $c>0$, call it $c_1$, so that $f_1 > f_2$ at $c_1$, and there 
exists another value of $c>c_1$, call it $c_2$, so that $f_1
< f_2$ at $c_2$.  
We also show that $+\infty>f_1,f_2>2$ for all values of $c \in 
(c_1,c_2)$.  
Then, by the continuity of $f_i$ and the intermediate value theorem, we
conclude that there exists a $c \in [c_1,c_2]$ such that $f_1 = f_2 > 2$.  

Furthermore,
by continuity, for any other value of $a$ sufficiently close to our chosen
value of $a$, there also exists a positive value for $c$ so that $f_1 = 
f_2>2$,
and hence the genus 1 catenoid exists for all $a$ sufficiently close to our 
chosen value of $a$.  Thus, with our method, showing
existence for one value of $a$ is sufficient to conclude the existence
of a one-parameter family of genus 1 catenoid cousins.  (However, we
cannot draw any conclusions about the possible 
range of the parameter $a$ for this
one-parameter family.)  

Thus, to prove Theorem 1.2, it is sufficient to do the 
following:
\begin{itemize}
\item We choose a suitable value for $a$ and call it 
$a_0$.  Then we choose suitable values for $c_1>0$ and $c_2>c_1$.
\item Using the initial condition $F$=identity, and the value $a_0$ for 
$a$, and using the value $c_i$ for $c$, we evaluate the Runge-Kutta algorithm
approximation for the solution to Bryant's equation along the path 
$\alpha_j(t)$.  With each evaluation of the Runge-Kutta alorithm, we make the 
evaluation by both rounding each mathematical operation upward and rounding 
each mathematical operation downward.  Thus for each output of the 
algorithm, we can find a range in which the theoretical value of the 
output of the algorithm must lie.
\item We then use Lemma 6.1, which gives an upper bound on the absolute 
value of the difference between the theoretical value of the 
output of the algorithm and the actual value of the solution of Bryant's 
equation.  Using Lemma 6.1, we can find a single bound which is valid 
for $a=a_0$ and all $c \in [c_1,c_2]$.
\item  We then have enough information to determine that any possible 
approximation errors are small enough to ensure that 
at $c_1$, $f_1 > f_2$ and that at $c_2$, $f_1
< f_2$.  
\item Then, it only remains to show that $f_1$ and $f_2$ are both 
bounded and greater than 2 for all $c \in [c_1,c_2]$.  We do this by 
showing that the derivative with respect to $c$ of the theoretical value 
of the output of the algorithm is bounded by a certain constant, 
for all $c \in [c_1,c_2]$.  This is the purpose of 
Lemma 6.2 -- it allows
us to place limits on the rate at which the output of the algorithm can 
change with respect to $c$.  This enables us to conclude that $2< f_i < 
\infty$ for all $c \in [c_1,c_2]$ simply by checking that this is so 
at a finite number of values of $c$ in $[c_1,c_2]$.  
\end{itemize}

\begin{lemma}
Let $\alpha(t), t \in [0,1]$ be a path in the complex plane.  Let 
\[F(\alpha(t)) = \left(\begin{array}{cc}
                            A(t) & B(t) \\
                    C(t) & D(t)
                \end{array}\right) \]
be an $SL(2,\bfC)$-valued function on $\alpha(t)$ 
such that $F(\alpha(0))$=identity and $F(\alpha(t))$ 
satisfies the equation 
\[  \left(\begin{array}{cc}
                            \frac{dA}{dt} & \frac{dB}{dt} \\
                    \frac{dC}{dt} & \frac{dD}{dt}
                \end{array}\right) = \left(\begin{array}{cc}
                            A & B \\
                    C & D
                \end{array}\right)\left(\begin{array}{cc}
                            ch_1 & ch_3 \\
                    ch_2 & ch_4
                \end{array}\right) \; , \] where $c$ is a real positive 
                constant and 
     $h_i$ are functions on the complex plane satisfying the bounds  
$|h_i| < M$, $|h_i^\prime| < M_1$, $|h_i^{\prime\prime}| < M_2$,
 $|h_i^{\prime\prime\prime}| < M_3$ on $\alpha(t)$ for $i=1,2,3,4$.  Assume 
 also that $h_1$ and $h_4$ are constant functions, and choose $n \in 
 \bfZ^+$ so that 
$\frac{Mc}{n} < \frac{1}{100}$.   Applying the 
 standard Runge-Kutta algorithm on the interval $t \in [0,1]$, using 
 $n$ equally lengthed intervals, 
 let the resulting approximate value for $F(\alpha(1))$ 
 produced by the Runge-Kutta algorithm be denoted by the matrix
\[\left(\begin{array}{cc}
                            \tilde{A} & \tilde{B} \\
                    \tilde{C} & \tilde{D}
                \end{array}\right) \; . \]
Then $|A(1)-\tilde{A}|$, $|B(1)-\tilde{B}|$, 
$|C(1)-\tilde{C}|$, and $|D(1)-\tilde{D}|$ are all bounded by
$\frac{e^{2.1cM}+e^{4.1cM}}{4.2cMn^{12}} \zeta$, 
where $\zeta$ is the following polynomial:
\[ \zeta(c,n,M,M_1,M_2,M_3) = 
\frac{n^9c}{72}(96c^3M^4+144c^2M^2M_1+18cM_1^2+48cMM_2+13M_3)+
\]\[\frac{n^8c^2}{48}(32c^2M^3M_1+12cMM_1^2+
             16cM^2M_2+4M_1M_2+11MM_3)+
\]\[
\frac{n^7c^2}{384}(80c^2M^2M_1^2+8cM_1^3+
             96c^2M^3M_2+64cMM_1M_2+5M_2^2+79cM^2M_3+22M_1M_3)+
\]\[\frac{n^6c^2}{2304}(48c^2MM_1^3+336c^2M^2M_1M_2+
             48cM_1^2M_2+60cMM_2^2+272c^2M^3M_3+236cMM_1M_3+
\]\[
49M_2M_3)+
\frac{n^5c^2}{2304}(48c^2MM_1^2M_2+
\]\[
54c^2M^2M_2^2+
             15cM_1M_2^2+200c^2M^2M_1M_3+26cM_1^2M_3+84cMM_2M_3+12M_3^2)+
\]\[\frac{n^4c^3}{13824}(90cMM_1M_2^2+9M_2^3+
             156cMM_1^2M_3+450cM^2M_2M_3+105M_1M_2M_3+116MM_3^3)+
\]\[
\frac{n^3c^3}{27648}(18cMM_2^3+210cMM_1M_2M_3+
             33M_2^2M_3+216cM^2M_3^2+44M_1M_3^2)+
\]\[
\]\[\frac{n^2c^3}{27648}(33cMM_2^2M_3+44cMM_1M_3^2+13M_2M_3^2)+
\frac{nc^3}{82944}(39cMM_2M_3^2+4M_3^3)+
\frac{c^4}{20736}MM_3^3 \; \; . \]
\end{lemma}

We include the condition that $h_1$ and $h_4$ are constant 
in Lemma 6.1, because this is sufficient for our application, and this 
will later allow 
us to assume a smaller lower bound for $n$.  It also simplifies the 
 proof somewhat.  However, it is not necessary to assume $h_1$ and $h_4$ 
 are constant in order to produce a lemma of this type.

\begin{proof}
The system of equations in the lemma can be
separated into two systems of two equations each -- one the systems
with variables $A$ and $B$, and the other with $C$ and $D$.  We
consider now the system involving $A$ and $B$:
\[ \frac{dA}{dt} = c h_1 A + 
              c h_2 B \; \; , \; 
\; \frac{dB}{dt} = c h_3 A + c h_4 
B \; \; . \]

Since $|h_i| < M$ for all $t \in [0,1]$ and all 
$i = 1,2,3,4$, we conclude that 
\[ \left|\frac{dA}{dt}\right| \leq cM |A| + cM |B| \; \; , \; \; 
\left|\frac{dB}{dt}\right| \leq cM |A| + cM |B| \; \; . \]
If we replace the inequalities in these equations by equalities, we
would be able to evaluate the system explicitly with 
%and the solution
%would be 
%\[                \left(\begin{array}{cc}
%                            A(t) \\
%                    B(t)
%                \end{array}\right) = \left(\exp{ \left(\begin{array}{cc}
%                            tcM & tcM \\
%                    tcM & tcM
%                \end{array}\right)}\right) \cdot \left(\begin{array}{cc}
%                            A(0) \\
%                    B(0)
%                \end{array}\right) \; \; , \] and from this we would
%conclude, since 
$A(0) = 1$ and $B(0) = 0$.  
%, that 
%\[ A(t) = 1 + \frac{1}{2} 
%\sum_{j=1}^\infty \frac{(2tcM)^{j}}{j!} \; \; , \; \; \; 
%B(t) = \frac{1}{2} \sum_{j=1}^\infty \frac{(2tcM)^{j}}{j!} \; \; . \]
%Since we actually only have inequalities in the above system, we can only
%conclude that 
It follows that 
\[ |A(t)| \leq 1 + \frac{1}{2} 
\sum_{j=1}^\infty \frac{(2tcM)^{j}}{j!} 
= \frac{1}{2} + \frac{1}{2} e^{2tcM} 
\; \; , \; \; \; \; |B(t)| \leq \frac{1}{2} 
\sum_{j=1}^\infty \frac{(2tcM)^{j}}{j!} 
= - \frac{1}{2} + \frac{1}{2} e^{2tcM} 
\; \; . \]

Now we run the standard Runge-Kutta algorithm on $t \in [0,1]$ 
for a system of two equations 
with $n$ steps of equal size $\frac{1}{n}$.  
The initial conditions are $A_0 = 1$ and
$B_0 = 0$.  
The algorithm at step $k$ is this:
\[ k_0 = \frac{c}{n}(h_1(\frac{k}{n})A_k+h_2(\frac{k}{n})B_k) \; , \; \; 
m_0 = \frac{c}{n}(h_3(\frac{k}{n})A_k+h_4(\frac{k}{n})B_k) \; , \]
\[ k_1 = \frac{c}{n}(h_1(\frac{k}{n}+\frac{1}{2n})(A_k+\frac{1}{2}k_0)+
           h_2(\frac{k}{n}+\frac{1}{2n})(B_k+\frac{1}{2}m_0)) \; , \]
\[ m_1 = \frac{c}{n}(h_3(\frac{k}{n}+\frac{1}{2n})(A_k+\frac{1}{2}k_0)+
           h_4(\frac{k}{n}+\frac{1}{2n})(B_k+\frac{1}{2}m_0)) \; , \]
\[ k_2 = \frac{c}{n}(h_1(\frac{k}{n}+\frac{1}{2n})(A_k+\frac{1}{2}k_1)+
           h_2(\frac{k}{n}+\frac{1}{2n})(B_k+\frac{1}{2}m_1)) \; , \]
\[ m_2 = \frac{c}{n}(h_3(\frac{k}{n}+\frac{1}{2n})(A_k+\frac{1}{2}k_1)+
           h_4(\frac{k}{n}+\frac{1}{2n})(B_k+\frac{1}{2}m_1)) \; , \]
\[ k_3 = \frac{c}{n}(h_1(\frac{k}{n}+\frac{1}{n})(A_k+k_2)+
           h_2(\frac{k}{n}+\frac{1}{n})(B_k+m_2)) \; , \]
\[ m_3 = \frac{c}{n}(h_3(\frac{k}{n}+\frac{1}{n})(A_k+k_2)+
           h_4(\frac{k}{n}+\frac{1}{n})(B_k+m_2)) \; , \]
\[ A_{k+1} = A_k + \frac{1}{6}(k_0 + 2k_1 + 2k_2 + k_3) \; , \; \; 
B_{k+1} = B_k + \frac{1}{6}(m_0 + 2m_1 + 2m_2 + m_3) \; . \]
We define the local discretization errors for $A$
and $B$ to be 
\[ d_{k+1}^A := A(\frac{k+1}{n}) - A(\frac{k}{n}) - 
\frac{1}{6}(\hat{k}_0 + 2\hat{k}_1 + 2\hat{k}_2 + \hat{k}_3) 
\; \; ,\] 
\[ d_{k+1}^B := B(\frac{k+1}{n}) - B(\frac{k}{n}) - 
\frac{1}{6}(\hat{m}_0 + 2\hat{m}_1 + 2\hat{m}_2 + \hat{m}_3)
\; \; ,\] where 
\[ \hat{k}_0 = \frac{c}{n}(h_1(\frac{k}{n})A(\frac{k}{n})+
h_2(\frac{k}{n})B(\frac{k}{n})) \; , \; \; 
\hat{m}_0 = \frac{c}{n}(h_3(\frac{k}{n})A(\frac{k}{n})+h_4(\frac{k}{n})B(\frac{k}{n})) \; , \]
\[ \hat{k}_1 = \frac{c}{n}(h_1(\frac{k}{n}+\frac{1}{2n})(A(\frac{k}{n})+\frac{1}{2}\hat{k}_0)+
h_2(\frac{k}{n}+\frac{1}{2n})(B(\frac{k}{n})+\frac{1}{2}\hat{m}_0)) \; , \]
and $\hat{m}_1, \hat{k}_2, \hat{m}_2, \hat{k}_3, \hat{m}_3$ are 
defined similarly, analogous to the way $m_1, k_2, m_2, k_3, m_3$ 
were defined.  
%\[ \hat{m}_1 = \frac{c}{n}(h_3(\frac{k}{n}+\frac{1}{2n})(A(\frac{k}{n})+\frac{1}{2}\hat{k}_0)+
%h_4(\frac{k}{n}+\frac{1}{2n})(B(\frac{k}{n})+\frac{1}{2}\hat{m}_0)) \; , \]
%\[ \hat{k}_2 = \frac{c}{n}(h_1(\frac{k}{n}+\frac{1}{2n})(A(\frac{k}{n})+\frac{1}{2}\hat{k}_1)+
% h_2(\frac{k}{n}+\frac{1}{2n})(B(\frac{k}{n})+\frac{1}{2}\hat{m}_1)) \; , \]
%\[ \hat{m}_2 = \frac{c}{n}(h_3(\frac{k}{n}+\frac{1}{2n})(A(\frac{k}{n})+\frac{1}{2}\hat{k}_1)+
%h_4(\frac{k}{n}+\frac{1}{2n})(B(\frac{k}{n})+\frac{1}{2}\hat{m}_1)) \; , \]
%\[ \hat{k}_3 = \frac{c}{n}(h_1(\frac{k}{n}+\frac{1}{n})(A(\frac{k}{n})+\hat{k}_2)+
%           h_2(\frac{k}{n}+\frac{1}{n})(B(\frac{k}{n})+\hat{m}_2)) \; , \]
%\[ \hat{m}_3 = \frac{c}{n}(h_3(\frac{k}{n}+\frac{1}{n})(A(\frac{k}{n})+\hat{k}_2)+
%           h_4(\frac{k}{n}+\frac{1}{n})(B(\frac{k}{n})+\hat{m}_2)) \; .\]
We define the maximums of the local discretization errors by 
\[ D^A := \max_k{|d_k^A|} \; \; , \; \; D^B := \max_k{|d_k^B|} \; \; ,
\; \; D := \max(D^A,D^B) \; \; .
\] 
We define the global discretization errors by 
\[ g_k^A := A(\frac{k}{n}) - A_k \; \; , \; \; 
g_k^B := B(\frac{k}{n}) - B_k \; \; , \; \; 
g_k := \max(|g_k^A|,|g_k^B|) \; \; .
\]
Since $A(\frac{k+1}{n}) = A(\frac{k}{n}) + 
\frac{1}{6}(\hat{k}_0 + 2\hat{k}_1 + 2\hat{k}_2 + \hat{k}_3)
+ d_{k+1}^A$,
we have that 
\[ g_{k+1}^A = g_{k}^A + 
\frac{1}{6}(\hat{k}_0 + 2\hat{k}_1 + 2\hat{k}_2 + \hat{k}_3)
-\frac{1}{6}(k_0 + 2k_1 + 2k_2 + k_3)
+ d_{k+1}^A \; \; , \] and therefore we can compute that 
\[ |g_{k+1}^A| \leq |g_{k}^A| + 
(\frac{cM}{n}+\frac{c^2M^2}{n^2}+\frac{2c^3M^3}{3n^3}+\frac{c^4M^4}{3n^4})
|A(\frac{k}{n}) - A_k| + \]\[
(\frac{cM}{n}+\frac{c^2M^2}{n^2}+\frac{2c^3M^3}{3n^3}+\frac{c^4M^4}{3n^4})
|B(\frac{k}{n}) - B_k| + |d_{k+1}^A| \; \; . \]
We assumed that $n > 100cM$, so we have 
\[ |g_{k+1}^A| \leq (1 + \frac{1.05cM}{n}) |g_{k}^A| + \frac{1.05cM}{n} 
|g_{k}^B| + D^A \; \; . \]
Similarly, 
\[ |g_{k+1}^B| \leq \frac{1.05cM}{n} |g_{k}^A| + (1 + \frac{1.05cM}{n}) 
|g_{k}^B| + D^B \; \; . \]
Thus, 
\[ g_{k+1} \leq (1 + \frac{2.1cM}{n}) g_{k} + D \; \; . \]
By repeated application of this inequality we have 
\[ g_n \leq (1 + \frac{2.1cM}{n})^n g_0 + 
\frac{(1 + \frac{2.1cM}{n})^n - 1}{\frac{2.1cM}{n}} D \; \; .\]
And since $g_0 = 0$, we have 
\[ g_n \leq \frac{e^{\frac{2.1cMn}{n}} - 1}{\frac{2.1cM}{n}} D 
< \frac{n e^{2.1cM}}{2.1cM} D\; \; . \]
Here we have used the fact that $e^x$ is convex on $\bfR$ and therefore $1+x
\leq e^x$ and therefore also $(1+x)^n \leq (e^x)^n = e^{xn}$ for any 
positive $x$.  

Note that $h_i(\frac{k}{n}+\frac{1}{2n})$ and 
$h_i(\frac{k}{n}+\frac{1}{n})$ and 
$A(\frac{k}{n}+\frac{1}{n})$ have the following Taylor expansions:
\[ h_i(\frac{k}{n}+\frac{1}{2n}) = h_i(\frac{k}{n}) + 
\frac{1}{2n} h_i^\prime(\frac{k}{n})
+ \frac{1}{8n^2} h_i^{\prime\prime}(\frac{k}{n})
+ \frac{1}{48n^3} h_i^{\prime\prime\prime}(\frac{k}{n}+\theta \frac{1}{2n})
\; , \]
\[ h_i(\frac{k}{n}+\frac{1}{n}) = h_i(\frac{k}{n}) + 
\frac{1}{n} h_i^\prime(\frac{k}{n})
+ \frac{1}{2n^2} h_i^{\prime\prime}(\frac{k}{n})
+ \frac{1}{6n^3} h_i^{\prime\prime\prime}(\frac{k}{n}+\theta \frac{1}{n})
\; , \] 
\[ A(\frac{k}{n}+\frac{1}{n}) = A(\frac{k}{n}) + \frac{1}{n} A^\prime(\frac{k}{n})
+ \frac{1}{2n^2} A^{\prime\prime}(\frac{k}{n})
+ \frac{1}{6n^3} A^{\prime\prime\prime}(\frac{k}{n})
+ \frac{1}{24n^4} A^{\prime\prime\prime\prime}(\frac{k}{n}+\theta \frac{1}{n})
\; , \]  for some values of $\theta \in [0,1]$.  Here the symbol $\prime$ 
denotes derivative with respect to $t$.  

Repeatedly using that 
$A^\prime = c h_1 A + c h_2 B$ and 
$B^\prime = c h_3 A + c h_4 B$, the above Taylor expansion for 
$A(\frac{k}{n}+\frac{1}{n})$ can be rewritten in a longer form 
so that it does not contain 
any terms of the form $A^{\prime}$, $B^{\prime}$, 
$A^{\prime\prime}$, $B^{\prime\prime}$, 
$A^{\prime\prime\prime}$, $B^{\prime\prime\prime}$, 
$A^{\prime\prime\prime\prime}$, or $B^{\prime\prime\prime\prime}$.  
Using this longer form for $A(\frac{k}{n}+\frac{1}{n})$, and using the 
above Taylor expansions for 
$h_i(\frac{k}{n}+\frac{1}{2n})$ and 
$h_i(\frac{k}{n}+\frac{1}{n})$, 
we can make a direct (but long) calculation to determine $d_{k+1}^A$ and 
$d_{k+1}^B$ in terms of $A$, $B$, $n$, $c$, $h_i$, and the 
derivatives (up to third order) 
of $h_i$.  These formulas are extremely long, so we 
do not include them here.  However, 
for each of these formulas we can take the sum of the absolute 
values of all of the terms, and make the following replacements:
$|h_i|$ by its upper bound $M$, $|h_i^{\prime}|$ 
by its upper bound $M_1$, $|h_i^{\prime\prime}|$ 
by its upper bound $M_2$, $|h_i^{\prime\prime\prime}|$ 
by its upper bound $M_3$, and $|A|$ and $|B|$ by their 
upper bound $\frac{1+e^{2cM}}{2}$.  We then 
get upper bounds for $|d_{k+1}^A|$ and 
$|d_{k+1}^B|$.  We can then find that a sufficient upper bound for both
$|d_{k+1}^A|$ and 
$|d_{k+1}^B|$ is $D \leq \frac{1+e^{2cM}}{2n^{13}}\zeta$.  
So we have that 
\[ g_n < \frac{ne^{2.1cM}}{2.1cM} D 
\leq \frac{ne^{2.1cM}}{2.1cM} \frac{1+e^{2cM}}{2n^{13}}\zeta
 \; . \]  

An identical 
argument gives the same conclusion for $C$ and $D$.
\end{proof}

\begin{lemma}
Suppose that the conditions of Lemma 6.1 hold for all $c$ 
contained in some interval $[c_1,c_2]$.  Then 
$|\frac{\partial \tilde{A}}{\partial c}|$, 
$|\frac{\partial \tilde{B}}{\partial c}|$, 
$|\frac{\partial \tilde{C}}{\partial c}|$, and 
$|\frac{\partial \tilde{D}}{\partial c}|$ are all bounded by 
$2.48Me^{2.4Mc}$ for all $c \in [c_1,c_2]$.
\end{lemma}

\begin{proof}
We consider the system of two equations for $A$ and $B$ here.  The case 
for $C$ and $D$ is identical.

Recall that the Runge-Kutta algorithm is 
\[ A_{j+1} = A_j + \frac{1}{6}(k_0 + 2k_1 + 2k_2 + k_3), \; \; 
B_{j+1} = B_j + \frac{1}{6}(m_0 + 2m_1 + 2m_2 + m_3) \; . \]
%where $k_0,k_1,k_2,k_3,m_0,m_1,m_2,m_3$ were defined in the last proof.
We can expand out the terms of these two equations so that everything is 
written in terms of only $A_j, B_j, A_{j+1}, B_{j+1}$ and 
$h_i, c, n$.  We then have
\[ A_{j+1} = A_j + {\cal Q} A_j + {\cal Q} B_j \; , \; \; 
B_{j+1} = B_j + {\cal Q} A_j + {\cal Q} B_j \; , \]
where $\cal Q$ is a polynomial that consists of the sum of 
two terms of the form 
$\frac{ch_*(*)}{6n}$, two terms of the form $\frac{ch_*(*)}{3n}$, 
six terms of the form $\frac{c^2h_*(*)h_*(*)}{6n^2}$, 
eight terms of the form $\frac{c^3h_*(*)h_*(*)h_*(*)}{12n^3}$, 
and eight terms of the form 
$\frac{c^4h_*(*)h_*(*)h_*(*)h_*(*)}{24n^4}$.  We will use the 
symbol $\cal Q$ to notate any polynomial of this form, regardless of 
what the subindices are for the functions $h_i(z)$ and regardless of at 
which 
value of $z(t)$ we are evaluating the functions $h_i(z)$.  (It is for this
reason that we are notating the functions $h_i(z)$ merely as $h_*(*)$.)
Although $\cal Q$ is not a well-defined notation, for our 
purposes this level of notation will be sufficient.  
It follows from the assumptions $|h_i| < M$ and $c>0$ and $\frac{Mc}{n} < 
\frac{1}{100}$ that $|{\cal Q}| < 1.2\frac{Mc}{n}$ and 
$|\frac{\partial {\cal Q}}{\partial c}| < 
\frac{M}{n}(1+2.4\frac{Mc}{n})$, regardless of what the indices of $h_i$ 
are and regardless of at which values of $z(t)$ we evaluate the functions 
$h_i$.

Applying the Runge-Kutta algorithm on $n$ equally lengthed steps, we find 
that the resulting estimates for $A$ and $B$ at $\alpha(1)$ are
of the form
\[ A_n = A_0 + p({\cal Q}) A_0 + p({\cal Q}) B_0 \; , \; \; 
B_n = B_0 + p({\cal Q}) A_0 + p({\cal Q}) B_0 \; , \]
where $p({\cal Q})$ is a polynomial in $\cal Q$ with 
$\frac{2^{j-1}n(n-1)(n-2)\cdot\cdot\cdot(n-j+1)}{j!}$ terms of the form 
${\cal Q}^j$ for each $j=1,...,n$.  Thus is follows that 
\[ \left| \frac{\partial A_n}{\partial c} \right|, 
\left| \frac{\partial B_n}{\partial c} \right| 
\leq 2 \sum_{j=1}^{n} 
\frac{2^{j-1}n(n-1)(n-2)\cdot\cdot\cdot(n-j+1)}{j!} 
\left|\frac{\partial}{\partial c}({\cal Q}^j)\right| \leq \]\[
\sum_{j=1}^{n} 
\frac{2^{j}n(n-1)(n-2)\cdot\cdot\cdot(n-j+1)}{j!} 
j \left|{\cal Q}\right|^{j-1} 
\left|\frac{\partial {\cal Q}}{\partial c}\right|
\leq \]\[ \sum_{j=1}^{n} 
\frac{2^{j}n(n-1)(n-2)\cdot\cdot\cdot(n-j+1)}{j!} 
j (1.2\frac{Mc}{n})^{j-1} (\frac{M}{n}(1+2.4\frac{Mc}{n})) \leq \]\[
\sum_{j=1}^{n} 
\frac{2^{j}}{(j-1)!} 
(1.2Mc)^{j-1} (M(1+2.4\frac{Mc}{n})) \leq
2 M (1+2.4\frac{Mc}{n}) \sum_{j=1}^{n} 
\frac{1}{(j-1)!} (2.4Mc)^{j-1} < \]\[
2.48 M \sum_{j=0}^{n-1} \frac{1}{j!} (2.4Mc)^{j} <
2.48 M \sum_{j=0}^{\infty} \frac{1}{j!} 
(2.4Mc)^{j} = 2.48 M e^{2.4Mc} \; . \]
\end{proof}

We are now in a position to prove Theorem 1.2:

\begin{proof}
Recall that our surface is described by the equation 
\[  \left(\begin{array}{cc}
                            \frac{dA}{dz} & \frac{dB}{dz} \\
                    \frac{dC}{dz} & \frac{dD}{dz}
                \end{array}\right) = \left(\begin{array}{cc}
                            A & B \\
                    C & D
                \end{array}\right)\left(\begin{array}{cc}
                            g & -g^2 \\
                    1 & -g
                \end{array}\right) \frac{c}{g} \; , \]
where $g = \sqrt{\frac{(z+1)(z-a)}{(z-1)(z+a)}}$.
This system separates into two systems -- one 
involving $A$ and $B$, and the other involving $C$ and $D$.  We
consider now the system involving $A$ and $B$:
\[ \frac{dA}{dz} = c A + \frac{c}{g} B \; \; , \; \; 
\frac{dB}{dz} = - c g A - c B \; \; . \]
Note that if we wish to evaluate this system along a curve from $z_0$
to $z_1$, linearly defined as $\frac{b-t}{b-a}z_0 + \frac{t-a}{b-a}
z_1, t \in [a,b]$, then
$\frac{dz}{dt} = \frac{z_1 - z_0}{b-a}$ for all 
$t \in [a,b]$.  So our system can
then be written as
\[ \frac{dA}{dt} = c \frac{z_1-z_0}{b-a} A + 
              c \frac{z_1-z_0}{b-a}\frac{1}{g} B \; \; , \; 
\; \frac{dB}{dt} = - c \frac{z_1-z_0}{b-a} g A - c \frac{z_1-z_0}{b-a} 
B \; \; . \]
Let $h_1 = \frac{z_1-z_0}{b-a}$, $h_2 = \frac{z_1-z_0}{b-a}
 \frac{1}{g}$, $h_3 =  -
\frac{z_1-z_0}{b-a} g$, and $h_4 =  - \frac{z_1-z_0}{b-a}$.

Now we choose the paths $\alpha_1(t)$ and $\alpha_2(t)$ for $t \in [0,1]$.  
The paths will start at the point $\alpha_1(0) = \alpha_2(0) = (0,1)$ in 
the base Reimann surface ${\cal M}_{a}^2$ and will 
be defined by their $z$ 
coordinates.  The $z$ coordinates of the paths will be polygonal and $t$ 
will be defined linearly with respect to $z$-length on each line segment.  
The path $\alpha_1(t)$  will project to a line segment from
$z=0$ $(t=0)$ to $z=1+0.4i$ $(t=0.67)$, then a line segment from $z=1+0.4i$ 
$(t=0.67)$ to $z=\frac{1}{2}(1+a)$ $(t=1)$.
The path $\alpha_2(t)$ will project to a line segment from
$z=0$ $(t=0)$ to $z=(a+0.2)+0.7i$ $(t=0.686)$, then a line segment from 
$z=(a+0.2)+0.7i$ $(t=0.686)$ to $z=a+\frac{1}{2}$ $(t=1)$.

We now solve Bryants differential equation along $\alpha_j(t)$.  At the 
beginning point 
$(z,w)=(0,1)$, that is, at $t=0$, the initial condition 
will be $F=$identity.
Suppose that 
the true value of $F$ at the endpoints is 
\[               F(\alpha_j(1)) = \left(\begin{array}{cc}
                            A_j & B_j\\
                    C_j & D_j
                \end{array}\right) \; , \; \; \; j = 1,2  \; , \]
and suppose that 
the approximate value of $F$ at the endpoints produced by the Runge-Kutta 
algorithm using $n$ steps of equal length is 
\[               \left(\begin{array}{cc}
                            \tilde{A}_j & \tilde{B}_j\\
                    \tilde{C}_j & \tilde{D}_j
                \end{array}\right) \; , \; \; \; j = 1,2  \; . \]
Of course, the exact values of $\tilde{A}_j, \tilde{B}_j, \tilde{C}_j, 
\tilde{D}_j$ cannot be computed on a computer, but by considering the 
possible round-off error for each mathematical operation in the algorithm, 
and keeping track of the possible cumulative round-off error, we can find 
intervals in which the values $\tilde{A}_j, \tilde{B}_j, \tilde{C}_j, 
\tilde{D}_j$ must lie.  That is, we can find real values 
$\tilde{A}_j^{ur}, \tilde{A}_j^{ui}, \tilde{A}_j^{lr}, \tilde{A}_j^{li}, 
\tilde{B}_j^{ur}, \tilde{B}_j^{ui}, \tilde{B}_j^{lr}, \tilde{B}_j^{li}, 
\tilde{C}_j^{ur}, \tilde{C}_j^{ui}, \tilde{C}_j^{lr}, \tilde{C}_j^{li}, 
\tilde{D}_j^{ur}, \tilde{D}_j^{ui}, \tilde{D}_j^{lr}, \tilde{D}_j^{li}$ 
such that 
\[ \tilde{{\cal I}}_j^{lr} \leq \mbox{Re}(\tilde{{\cal I}}_j) \leq 
\tilde{{\cal I}}_j^{ur} \; , \; \; \tilde{{\cal I}}_j^{li} \leq 
\mbox{Im}(\tilde{{\cal I}}_j) \leq \tilde{{\cal I}}_j^{ui} \; , \] 
%\[ \tilde{B}_j^{lr} \leq \mbox{Re}(\tilde{B}_j) \leq 
%\tilde{B}_j^{ur} \; , \; \; \tilde{B}_j^{li} \leq \mbox{Im}(\tilde{B}_j) \leq 
%\tilde{B}_j^{ui} \; , \] 
%\[ \tilde{C}_j^{lr} \leq \mbox{Re}(\tilde{C}_j) \leq 
%\tilde{C}_j^{ur} \; , \; \; \tilde{C}_j^{li} \leq \mbox{Im}(\tilde{C}_j) \leq 
%\tilde{C}_j^{ui} \; , \] 
%\[ \tilde{D}_j^{lr} \leq \mbox{Re}(\tilde{D}_j) \leq 
%\tilde{D}_j^{ur} \; , \; \; \tilde{D}_j^{li} \leq \mbox{Im}(\tilde{D}_j) \leq 
%\tilde{D}_j^{ui} \; , \] 
for ${\cal I} = A, B, C, D$ and 
$j=1,2$, thus giving us a total of 16 equations.  
(These bounds can be computed 
using code written in a programming language such as C++ or 
Fortran.  They also could be computed using a shorter code written in a 
scientific programming language such as CXSC or PROFIL.)  
To explain why we write the 
superscripts in this way, the first letter is either $u$ or $l$ to denote 
either $u$pper bound or $l$ower bound, and the second letter is either 
$r$ or $i$ to denote either $r$eal part or $i$maginary part.

Choosing $a$ to be 1.78, we find that we have the following 
bounds for both paths: $M=4.6$, $M_1 = 48$, 
$M_2 = 850$, $M_3 = 25000$.  (These bounds are defined in 
Lemma~6.1.)  Then, if we choose any 
$c \in [0.0495,0.0505]$, from Lemma 6.1 we see that if 
$n > 500$, the errors incurred by 
the Runge-Kutta algorithm on $A_j$, $B_j$, $C_j$, and $D_j$ 
are less than $\epsilon$=0.00001.  That is, 
$|\tilde{A}_j-A_j| < \epsilon$,
$|\tilde{B}_j-B_j| < \epsilon$,
$|\tilde{C}_j-C_j| < \epsilon$, and 
$|\tilde{D}_j-D_j| < \epsilon$, for all $c \in 
[0.0495,0.0505]$.  It follows that 
\[ \tilde{{\cal I}}_j^{lr}-\epsilon \leq \mbox{Re}({\cal I}_j) \leq 
\tilde{{\cal I}}_j^{ur}+\epsilon \; , \; \; 
\tilde{{\cal I}}_j^{li}-\epsilon 
\leq \mbox{Im}({\cal I}_j) \leq 
\tilde{{\cal I}}_j^{ui}+\epsilon \; , \] 
%\[ \tilde{B}_j^{lr}-\epsilon \leq \mbox{Re}(B_j) \leq 
%\tilde{B}_j^{ur}+\epsilon \; , \; \; \tilde{B}_j^{li}-\epsilon 
%\leq \mbox{Im}(B_j) \leq 
%\tilde{B}_j^{ui}+\epsilon \; , \] 
%\[ \tilde{C}_j^{lr}-\epsilon \leq \mbox{Re}(C_j) \leq 
%\tilde{C}_j^{ur}+\epsilon \; , \; \; \tilde{C}_j^{li}-\epsilon 
%\leq \mbox{Im}(C_j) \leq 
%\tilde{C}_j^{ui}+\epsilon \; , \] 
%\[ \tilde{D}_j^{lr}-\epsilon \leq \mbox{Re}(D_j) \leq 
%\tilde{D}_j^{ur}+\epsilon \; , \; \; \tilde{D}_j^{li}-\epsilon 
%\leq \mbox{Im}(D_j) \leq 
%\tilde{D}_j^{ui}+\epsilon \; , \] 
for ${\cal I} = A, B, C, D$ and $j=1,2$.

The value of $a$ is fixed, but $c$ can be any value in the range
$[0.0495,0.0505]$, so the numbers $A_{j}$, $\ldots$, $D_{j}$, 
$\tilde{A}_j$,
$\ldots$, $\tilde{D}_j$, $\tilde{A}_j^{ur}$, $\ldots$, $\tilde{D}_j^{li}$ 
are all
functions of $c$.  To clearly show this dependence, we shall denote
these numbers by $A_{j}(c)$, $\ldots$, $D_{j}(c)$, $\tilde{A}_j(c)$,
$\ldots$, $\tilde{D}_j(c)$, $\tilde{A}_j^{ur}(c)$, $\ldots$, 
$\tilde{D}_j^{li}(c)$ for the rest of this proof.

By Lemma 6.2, 
$|\frac{\partial \tilde{A}_j(c)}{\partial c}|$, 
$|\frac{\partial \tilde{B}_j(c)}{\partial c}|$, 
$|\frac{\partial \tilde{C}_j(c)}{\partial c}|$, and 
$|\frac{\partial \tilde{D}_j(c)}{\partial c}|$ are all bounded by 
$2.48Me^{2.4Mc} < 20$ for all $c \in [0.0495,0.0505]$.  So, it follows 
that if we choose any $c \in [0.04999,0.05001]$, 
we have that $\tilde{A}_j(c)$, $\tilde{B}_j(c)$, $\tilde{C}_j(c)$, and 
$\tilde{D}_j(c)$ can vary from their values at $c=0.05$ by at most 
$\hat{\epsilon}=$0.0002.  That is, 
$|\tilde{A}_j(c)-\tilde{A}_j(0.05)| < \hat{\epsilon}$,
$|\tilde{B}_j(c)-\tilde{B}_j(0.05)| < \hat{\epsilon}$,
$|\tilde{C}_j(c)-\tilde{C}_j(0.05)| < \hat{\epsilon}$, and 
$|\tilde{D}_j(c)-\tilde{D}_j(0.05)| < \hat{\epsilon}$, for all $c \in 
[0.04999,0.05001]$.  We conclude that 
\[ \tilde{{\cal I}}_j^{lr}(0.05)-\epsilon -\hat{\epsilon} \leq 
\mbox{Re}({\cal I}_j(c)) \leq 
\tilde{{\cal I}}_j^{ur}(0.05)+\epsilon +\hat{\epsilon} \; , 
\]\[ \; \; \tilde{{\cal I}}_j^{li}(0.05)-\epsilon 
-\hat{\epsilon} 
\leq \mbox{Im}({\cal I}_j(c)) \leq 
\tilde{{\cal I}}_j^{ui}(0.05)+\epsilon +\hat{\epsilon} \; , \] 
%\[ \tilde{B}_j^{lr}(0.05)-\epsilon -\hat{\epsilon} \leq 
%\mbox{Re}(B_j(c)) \leq 
%\tilde{B}_j^{ur}(0.05)+\epsilon +\hat{\epsilon} \; ,
%\]\[ \; \; \tilde{B}_j^{li}(0.05)-\epsilon 
%-\hat{\epsilon} 
%\leq \mbox{Im}(B_j(c)) \leq 
%\tilde{B}_j^{ui}(0.05)+\epsilon +\hat{\epsilon} \; , \] 
%\[ \tilde{C}_j^{lr}(0.05)-\epsilon -\hat{\epsilon} \leq \mbox{Re}(C_j(c)) \leq 
%\tilde{C}_j^{ur}(0.05)+\epsilon +\hat{\epsilon} \; ,
%\]\[ \; \; \tilde{C}_j^{li}(0.05)-\epsilon 
%-\hat{\epsilon} 
%\leq \mbox{Im}(C_j(c)) \leq 
%\tilde{C}_j^{ui}(0.05)+\epsilon +\hat{\epsilon} \; , \] 
%\[ \tilde{D}_j^{lr}(0.05)-\epsilon -\hat{\epsilon} \leq \mbox{Re}(D_j(c)) \leq 
%\tilde{D}_j^{ur}(0.05)+\epsilon +\hat{\epsilon} \; ,
%\]\[ \; \; \tilde{D}_j^{li}(0.05)-\epsilon -\hat{\epsilon} 
%\leq \mbox{Im}(D_j(c)) \leq 
%\tilde{D}_j^{ui}(0.05)+\epsilon +\hat{\epsilon} \; , \] 
for ${\cal I} = A, B, C, D$ and $j=1,2$ and all 
$c \in [0.04999,0.05001]$.  This is 
sufficient to conclude that both 
$f_1, f_2 \in (2,+\infty)$ for all $c \in [0.04999,0.05001]$.
Checking in this way on many small intervals (a finite number of 
intervals), we can conclude that 
$2<f_1,f_2<\infty$ for all $c \in [0.0495,0.0505]$.

Then, as we saw before, solving the period problem means solving
$f_1 = f_2 > 2$.  Running the Runge-Kutta algorithm with 
$c=0.0495$, we conclude that 
\[ \tilde{{\cal I}}_j^{lr}(0.0495)-\epsilon \leq 
\mbox{Re}({\cal I}_j(0.0495)) \leq 
\tilde{{\cal I}}_j^{ur}(0.0495)+\epsilon \; ,
\]\[ \; \; \tilde{{\cal I}}_j^{li}(0.0495)-\epsilon 
\leq \mbox{Im}({\cal I}_j(0.0495)) \leq 
\tilde{{\cal I}}_j^{ui}(0.0495)+\epsilon \; , \] 
%\[ \tilde{B}_j^{lr}(0.0495)-\epsilon \leq \mbox{Re}(B_j(0.0495)) \leq 
%\tilde{B}_j^{ur}(0.0495)+\epsilon \; ,
%\]\[ \; \; \tilde{B}_j^{li}(0.0495)-\epsilon 
%\leq \mbox{Im}(B_j(0.0495)) \leq 
%\tilde{B}_j^{ui}(0.0495)+\epsilon \; , \] 
%\[ \tilde{C}_j^{lr}(0.0495)-\epsilon \leq \mbox{Re}(C_j(0.0495)) \leq 
%\tilde{C}_j^{ur}(0.0495)+\epsilon \; ,
%\]\[ \; \; \tilde{C}_j^{li}(0.0495)-\epsilon 
%\leq \mbox{Im}(C_j(0.0495)) \leq 
%\tilde{C}_j^{ui}(0.0495)+\epsilon \; , \] 
%\[ \tilde{D}_j^{lr}(0.0495)-\epsilon \leq \mbox{Re}(D_j(0.0495)) \leq 
%\tilde{D}_j^{ur}(0.0495)+\epsilon \; ,
%\]\[ \; \; \tilde{D}_j^{li}(0.0495)-\epsilon 
%\leq \mbox{Im}(D_j(0.0495)) \leq 
%\tilde{D}_j^{ui}(0.0495)+\epsilon \; , \] 
for ${\cal I} = A, B, C, D$ and $j=1,2$.  
These estimates 
are sufficient to show that $f_1 > f_2$ at $c=0.0495$.  Similarly we can 
show that $f_1 < f_2$ at $c=0.0505$.  We conclude that there exists a 
value of $c \in [0.0495,0.0505]$ so that $f_1 = f_2 > 2$.  
We have thus shown of existence of at 
least one genus 1 catenoid cousin.  Then, since the problem is continuous 
in $a$, we know that for all $a$ sufficiently close to 1.78 there exists a 
positive value for $c$ so that $f_1 = f_2 > 2$.  This proves existence of a 
one-parameter family of genus 1 catenoid cousins.
\end{proof}

\end{document}